\documentclass[11pt, a4paper]{amsart}
\usepackage[utf8]{inputenc}
\usepackage[all,cmtip]{xy}%nEW
\usepackage{amsmath, amsfonts, amssymb, amsthm}

\usepackage{graphicx}
\usepackage[inline]{enumitem}
\usepackage{hyperref}
\usepackage{multicol}

\usepackage[cal=boondox]{mathalpha}

\theoremstyle{definition}
\newtheorem{definition}{Definition}%[subsection]
\newtheorem{remark}[definition]{Remark}

\theoremstyle{plain}

\newtheorem{lemma}[definition]{Lemma}
\newtheorem{theorem}[definition]{Theorem}
\newtheorem{corollary}[definition]{Corollary}

\usepackage{tikz, float, tikz-cd}
	\usetikzlibrary{arrows}

\usepackage[
    a4paper,
    top=2.5cm,
    bottom=2.5cm,
    left=2.5cm,
    right=2.5cm,
    includeheadfoot,
    headheight=14pt,
    footskip=1cm,
    marginparwidth=3cm,
    marginparsep=0.5cm
]{geometry}

\hypersetup{pdfborder={0 0 0}}

\newcommand{\inert}{\mathrm{inert}}

\newlist{steplist}{enumerate}{1}
\setlist[steplist]{label=\textbf{Step \arabic*}, align=left, wide=1pt}

\newlist{caselist}{enumerate}{1}
\setlist[caselist]{label=\textbf{Case \arabic*}, align=left, wide=1pt}

\newlist{termlist}{enumerate}{1}
\setlist[termlist]{label=\textbf{Term \arabic*}, align=left, wide=1pt}

\author{Benjamin Enriquez}
\address{\scriptsize Institut de Recherche Mathématique Avancée (UMR 7501), University of Strasbourg, 7 rue René Descartes, 67000 Strasbourg, France}
\email{b.enriquez@math.unistra.fr}

\author{Hidekazu Furusho}
\address{\scriptsize Graduate School of Mathematics, Nagoya University, Furo-cho, Chikusa-ku, Nagoya, 464-8602,
Japan}
\email{furusho@math.nagoya-u.ac.jp}

\date{July 10, 2026}

\title[A stabilizer interpretation of the Grothendieck-Teichm\"uller group $\mathsf{GRT}_1(\mathbf k)$]{
A stabilizer interpretation of the Grothendieck-Teichm\"uller group $\mathsf{GRT}_1(\mathbf k)$
}

\begin{document}
\begin{abstract}

If $\mathfrak u$ and $\mathfrak v$ are Lie algebras, then the product $\mathrm{Out}(\mathfrak u)\times \mathrm{Out}(\mathfrak v)$  
of their outer automorphism groups naturally acts on the set of outer Lie algebra morphisms from $\mathfrak u$ to $\mathfrak v$;
the stabilizer of the outer class of a given such morphism is then a subgroup of $\mathrm{Out}(\mathfrak u)\times 
\mathrm{Out}(\mathfrak v)$. We show that this leads to two related interpretation of the Grothendieck-Teichmüller group 
$\mathsf{GRT}_1(\mathbf k)$, where $\mathfrak u,\mathfrak v$ are the Lie algebra of infinitesimal braids on the plane 
(resp. framed infinitesimal braids on the sphere) with 3 and 4 (resp. 4 and 5) strands: namely, it can be expressed as  
the joint intersection of the stabilizer groups of the outer classes of certain strand doubling morphisms $\phi$ and $\psi$
with $\mathrm{Out}^*(\mathfrak u)\times\mathrm{Out}(\mathfrak v)$, where $\mathrm{Out}^*(\mathfrak u)$ is a subgroup of 
$\mathrm{Out}(\mathfrak u)$ of outer classes of inertia-preserving automorphisms of $\mathfrak u$.
% We propose two related interpretations of the Grothendieck-Teichmüller Lie algebra $\mathfrak{grt}_1$, 
% both associated with a pair $\mathfrak u,\mathfrak v$ of infinitesimal braid Lie algebras and a pair of 
% morphisms $\phi,\psi$ relating them. In the first (resp. second) case, $\mathfrak u$ and $\mathfrak v$ are the Lie algebras
% of infinitesimal braids on the plane (resp. framed infinitesimal braids on the 2-sphere) with $3$ and $4$ strands (resp. with $4$ and 
% $5$ strands), and $\phi$ and $\psi$ correspond to doublings of points. Namely, we identify in both cases  $\mathfrak{grt}_1$
% with the set of all outer tangential derivations of $\mathfrak u$, intertwined by both $\phi$ and $\psi$ to some outer 
% derivation of $\mathfrak v$. 
\end{abstract}

\maketitle

    {\footnotesize \tableofcontents}

\subsection{Introduction}\label{sect:intro}

In \cite{Dr}, Drinfeld introduced the related notions of associators and the Gro\-then\-dieck-Teichmüller (GT) $\mathbb Q$-group 
scheme. These notions were later used in many areas of mathematics (see a survey in the Introduction of \cite{W}). 

The GT group functor takes a commutative $\mathbb Q$-algebra with unit $\mathbf k$ to 
$\mathsf{GRT}_1(\mathbf k)$, defined in \cite{Dr} as an automorphism group of the collection of all 
symmetric monoidal categories over $\mathbf k$ which are infinitesimally braided (i.e. equipped with an infinitesimal 
braiding). As $\mathsf{GRT}_1$ is a pro-unipotent group scheme, it can uniquely be recovered from 
its Lie algebra $\mathfrak{grt}_1$. This Lie algebra is presented in \cite{Dr} by an explicit system of
linear conditions, called 2-cycle, hexagon and pentagon relations. 

Equivalent linear systems for the definition of $\mathfrak{grt}_1$ have been given in several works. 
In \cite{F1}, $\mathfrak{grt}_1$ was shown to be equal to the a priori larger set of solutions of the 
pentagon relation. In \cite{FHK}, it was shown to be equal to an a priori even larger set of solutions 
of a linear system. On a different note, an interpretation of $\mathfrak{grt}_1$ in terms of confluence relations
was obtained in \cite{HS,F2}. 

Identifications of $\mathfrak{grt}_1$ with explicit Lie algebras of derivations were also proposed. In \cite{BN}, 
$\mathfrak{grt}_1$ is interpreted as the Lie algebra of derivations of a `universal' infinitesimally braided 
symmetric monoidal category, denoted $\mathbf{PaCD}$ 
(parenthesized chord diagrams), compatible with the cabling operations of this category. The underlying operadic structure was 
later clarified in \cite{Fr}, which enabled the identification of $\mathsf{GRT}_1(\mathbf k)$ with 
an operadic automorphism group. 

An interpretation of $\mathfrak{grt}_1$ in terms of infinitesimal braid Lie algebras was obtained in \cite{I2}. 
There,  %independently of Drinfeld's work, 
a form over $\mathbb Z$ of this Lie algebra is identified with  
the `stable derivation algebra'
of the tower of infinitesimal braid Lie algebras on the sphere $(\mathfrak P_n)_{n\geq3}$. For $n\geq3$, 
$\mathfrak P_n=\mathrm{gr}\pi_1\mathrm{Cf}_n(S^2)$, where $\mathrm{gr}$ is the Lie algebra associated with the lower 
central series filtration of a group, and $\mathrm{Cf}_n(X)$ is the configuration space of $n$ distinct points on a space 
$X$; it contains elements $X_{ij}$ where $i\neq j\in[1,n]$ which correspond to the 
winding of braids $i$ and $j$. One defines $\mathfrak{der}^*(\mathfrak P_n)$ as the set of all derivations of $\mathfrak P_n$
which take each $X_{ij}$ to an element of $[X_{ij},\mathfrak P_n]$; this is a Lie subalgebra of  
$\mathfrak{der}(\mathfrak P_n)$, which also contains the Lie subalgebra $\mathfrak{inn}(\mathfrak P_n)$ of inner derivations 
as an ideal, and one defines 
$\mathfrak{out}^*(\mathfrak P_n)$ as the quotient $\mathfrak{der}^*(\mathfrak P_n)/\mathfrak{inn}(\mathfrak P_n)$, and finally 
$\mathcal D_n:=\mathfrak{out}^*(\mathfrak P_n)^{S_n}$, the action being induced by that of the symmetric group 
 $S_n$ on $\mathfrak P_n$ by permutation of indices. 
 
 There is a natural map $\mathcal D_{n+1}\to \mathcal D_n$, which is constructed as follows. 
Let $i\in[1,n+1]$. The erasing of strand $i$ induces a Lie algebra morphism 
$\mathrm{pr}_i : \mathfrak P_{n+1}\to\mathfrak P_n$. There is a Lie algebra morphism
$p_i : \mathfrak{der}^*(\mathfrak P_{n+1})\to\mathfrak{der}^*(\mathfrak P_n)$, such that 
$\mathrm{pr}_i(D(x))=p_i(D)(\mathrm{pr}_i(x))$ for any $x\in\mathfrak P_{n+1},D\in\mathfrak{der}^*(\mathfrak P_{n+1})$. 
The morphism $p_i$ induces a morphism 
$p_i : \mathfrak{out}^*(\mathfrak P_{n+1})\to\mathfrak{out}^*(\mathfrak P_n)$. 
The restriction of $p_i$ to $\mathcal D_{n+1}$ is independent of $i$ and its image is contained in  
$\mathcal D_n$; this is the map $\mathcal D_{n+1}\to\mathcal D_n$. 
The results of \cite{I2} can then be summarized as follows: 

\begin{theorem} (see \cite{I2})
The map $\mathcal D_{n+1}\to\mathcal D_n$ is an isomorphism if $n\geq5$ and is injective if $n=4$; moreover, 
there is an isomorphism $\mathcal D_5\otimes\mathbb Q\simeq\mathfrak{grt}_1$. 
\end{theorem}

From the geometric viewpoint, the key role in this result is therefore played by the tower of Lie algebras
$(\mathfrak P_n)_{n\geq4}$ together with the compatible actions of the symmetric groups, and with the 
`erasing of strands' morphisms $\mathrm{pr}_i : \mathfrak P_{n+1}\to \mathfrak P_n$, 
which is part of the operadic structure involved in \cite{Fr}.

The purpose of the present work is to exhibit two related alternative interpretations of the Lie 
algebra $\mathfrak{grt}_1$ in terms of infinitesimal braid Lie algebras. One of these interpretations 
is based on particular morphisms of doubling of strands from the tower of Lie algebra 
$(\mathfrak t_n)_{n\geq0}$, which turns out to be also part of the operadic structure involved in \cite{Fr}.

The first interpretation involves the Lie algebras of infinitesimal braids on the plane
$\mathfrak t_n:=\mathrm{gr}\pi_1\mathrm{Cf}_n(\mathbb R^2)$ ($n\geq 2$), which are related by the morphisms 
$\mathfrak t_n\to \mathfrak t_m$, $x\mapsto x^\phi$ attached to a partially defined map $\phi : [m]\supset S_\phi\to [n]$, 
which are obtained upon taking graded Lie algebras of fundamental groups from the diagram  
$\mathrm{Cf}_n(\mathbb C)\stackrel{\phi}{\to} \overline{\mathrm{Cf}}_m(\mathbb C)\hookleftarrow \mathrm{Cf}_m(\mathbb C)$, 
the bar denoting the ASFM (Axelrod-Singer-Fulton-McPherson) compactification; where the map $\phi : \mathrm{Cf}_n(\mathbb C)\to 
\overline{\mathrm{Cf}}_m(\mathbb C)$ may be obtained as the limit for $t\to 0$ (in the sense of the Introduction to \cite{FM}) 
of the maps 
$\phi_t : \mathrm{Cf}_n(\mathbb C)\to \mathrm{Cf}_m(\mathbb C)$ (well-defined for $|t|$ small enough)
$(x_1,...,x_n)\mapsto (y^t_1,...,y^t_m)$, where $y_i^t:=
x_{\phi(i)}+tn(i)$ for $i\notin S_\phi$
and $y_i^t:=t^{-1}(1+n(i))$ for $i\in S_\phi$. 
where $n(i):=|[1,i-1]\cap\phi^{-1}\phi(i)|$ for $i\notin S_\phi$ and 
$n(i):=|[1,i-1]\cap S_\phi|$ for $i\in S_\phi$.  With $\mathfrak t_3,\mathfrak t_4,x\mapsto x^\phi$ 
as in \S\ref{sect:1:1} and $\mathfrak G,\mathfrak G_\inert$ as in \S\ref{sect:12}, the result is formulated as follows: 

\begin{theorem} \label{thm:2} (see Thm. \ref{thm:main})
The subset $\mathfrak{grt}_1\subset\mathfrak G$ is equal to the subset of $\mathfrak G_\inert$ of all elements  
$C$ such that there exists $D\in\mathfrak{der}(\mathfrak t_4)$ and $X\in\mathfrak t_4$, 
satisfying 

(a) the morphism $\mu_{123}:\mathfrak t_3\to\mathfrak t_4,x\mapsto x^{1,2,3}$ intertwines $D_C$ (see Lem. \ref{lem:def:DC}) 
and $D$, and 

(b) the morphism $\mu_{124}:\mathfrak t_3\to\mathfrak t_4,x\mapsto x^{1,2,4}$ intertwines $D_C$ and $D+\mathrm{ad}_X$. 
\end{theorem}

The second interpretation  is the analogue of the first one, where the plane is replaced by the 2-sphere. 
The involved Lie algebras are then those of {\it framed} infinitesimal braids on the 2-sphere
$\mathfrak P_{\vec n}:=\mathrm{gr}\pi_1\mathrm{Cf}_{\vec n}(S^2)$ ($n\geq 3$), where $\mathrm{Cf}_{\vec n}(S^2)$
is the set of $n$-uples $((x_1,v_1),...,(x_n,v_n))$ where $(x_1,...,x_n)\in \mathrm{Cf}_{n}(S^2)$ 
and $v_i\in T_{x_i}S^2\smallsetminus0$. The morphisms $\mathfrak P_{\vec n}\to \mathfrak P_{\vec m}$, $x\mapsto x^\phi$ 
are then attached to partially defined maps $\phi : [m]\supset S_\phi\to [n]$ and are obtained from the diagram 
$\mathrm{Cf}_{\vec n}(S^2)\stackrel{\phi}{\to} \overline{\mathrm{Cf}}_{\vec m}(S^2)\hookleftarrow \mathrm{Cf}_{\vec m}(S^2)$, 
where $\overline{\mathrm{Cf}}_{\vec m}(S^2)$ is the fibered product 
$\overline{\mathrm{Cf}}_{m}(S^2)\times_{(S^2)^m}(T^\times S^2)^m$, where 
$T^\times S^2:=\{(x,v)|x\in S^2,v\in T_xS^2\smallsetminus 0\}$ and the map 
$\phi : \mathrm{Cf}_{\vec n}(S^2)\to\overline{\mathrm{Cf}}_{\vec m}(S^2)$ is the limit for 
$t\to 0$ of the maps (well-defined for $|t|$ small enough)
$\mathrm{Cf}_{\vec n}(S^2)\to\mathrm{Cf}_{\vec m}(S^2)$
given by $(\mathbf x_1,...,\mathbf x_n)\mapsto (\mathbf y^t_1,...,\mathbf y^t_m)$, where 
$\mathbf y_i^t:=T\mathrm{exp}_{\mathbf x_{\phi(i)}}(tn(i)v_i,v_i)$, with $n(i):=|[1,i-1]\cap\phi^{-1}\phi(i)|$, 
the exponential map being relative to the round Riemannian metric on $S^2$
(where the exponential map attached to a Riemannian manifold $(M,g)$ and $x\in M$ is denoted 
$T_xM\supset U_x\stackrel{\mathrm{exp}_x}{\to} M$, and its tangent map by 
$T\mathrm{exp}_x : TU_x\to TM$).  With $\mathfrak P_{\vec 4},\mathfrak P_{\vec 5}$,  
$x\mapsto x^\phi$ as in §\ref{sec:2:1}, the result is formulated as follows: 

\begin{theorem} \label{thm:3}(see Thm. \ref{thm:main:2})
The subset $\mathfrak{grt}_1$ of $\mathfrak G$ 
is equal to the set of all elements $C\in\mathfrak G_\inert$, such that there exists a derivation 
$\vec D$ of $\mathfrak P_{\vec 5}$ and $\vec X\in\mathfrak P_{\vec 5}$, such that: 

\emph{\((a')\)} 
the morphism $\mathfrak P_{\vec 4}\to\mathfrak P_{\vec 5}$, $x\mapsto x^{1,2,3,45}$ intertwines $\vec D_C$ and $\vec D$, and 

\emph{\((b')\)} 
the morphism $\mathfrak P_{\vec 4}\to\mathfrak P_{\vec 5}$, $x\mapsto x^{1,2,4,35}$ intertwines $\vec D_C$ and 
$\vec D+\mathrm{ad}_{\vec X}$.
\end{theorem}

Thm. \ref{thm:2} leads to following interpretation of $\mathsf{GRT}_1(\mathbf k)$ in terms of stabilizers.  
Let $\mathbf k$ be a commutative $\mathbb Q$-algebra with unit.
For a positively graded complete $\mathbb Q$-Lie algebra $\mathfrak a$, 
generated in degree $1$ and whose degree 1 component
is finite dimensional, define $\mathrm{Aut}_1(\mathfrak a\hat\otimes \mathbf k)$
as the group of $\mathbf k$-linear Lie algebra automorphisms of  $\mathfrak a\hat\otimes \mathbf k$, 
filtered and with associated graded the identity. The assignment $g\mapsto \mathrm{Ad}_g$ is a group morphism  
$\mathrm{Ad} : \mathrm{exp}(\mathfrak a\hat\otimes \mathbf k)\to\mathrm{Aut}_1(\mathfrak a\hat\otimes \mathbf k)$, 
whose image is a normal subgroup. The corresponding quotient is then denoted  by 
$\mathrm{Out}_1(\mathfrak a\hat\otimes \mathbf k):=\mathrm{Aut}_1(\mathfrak a\hat\otimes \mathbf k)/\mathrm{Ad}(\mathrm{exp}(\mathfrak a\hat\otimes \mathbf k))$. 

For a Lie algebra $\mathfrak b$  with the same properties as $\mathfrak a$, 
define $\mathrm{Hom}(\mathfrak a,\mathfrak b)(\mathbf k)$ as the set 
of $\mathbf k$-linear morphisms of complete Lie algebras $\mathfrak a\hat\otimes\mathbf k\to \mathfrak b\hat\otimes\mathbf k$, and by 
$\mathrm{OutHom}(\mathfrak a,\mathfrak b)(\mathbf k)$ the quotient set 
$\mathrm{Hom}(\mathfrak a,\mathfrak b)(\mathbf k)/\mathrm{exp}(\mathfrak b\hat\otimes\mathbf k)$, where 
the group $\mathrm{exp}(\mathfrak b\hat\otimes\mathbf k)$ acts on 
$\mathrm{Hom}(\mathfrak a,\mathfrak b)(\mathbf k)$ by post-composition.
The group $\mathrm{Aut}_1(\mathfrak a\hat\otimes\mathbf k)\times\mathrm{Aut}_1(\mathfrak b\hat\otimes\mathbf k)$
acts on $\mathrm{Hom}(\mathfrak a,\mathfrak b)(\mathbf k)$ by combining pre- and post-composition. This 
induces an action of $\mathrm{Out}_1(\mathfrak a\hat\otimes \mathbf k)\times \mathrm{Out}_1(\mathfrak b\hat\otimes \mathbf k)$
on $\mathrm{OutHom}(\mathfrak a,\mathfrak b)(\mathbf k)$. 

For $\mu\in \mathrm{Hom}(\mathfrak a,\mathfrak b)(\mathbb Q)$, let $\overline\mu$ be its image in 
$\mathrm{OutHom}(\mathfrak a,\mathfrak b)(\mathbb Q)$, and also by $\overline\mu$ the image of the latter element in 
$\mathrm{OutHom}(\mathfrak a,\mathfrak b)(\mathbf k)$ for any $\mathbf k$. 

In Lem. \ref{lem:24}, we define a morphism $\mathfrak G_\inert\to \mathfrak{out}_1(\mathfrak t_3)$, 
which leads to a group morphism $\mathrm{exp}(\hat{\mathfrak G}_\inert\hat\otimes\mathbf k)
\to\mathrm{Out}_1(\hat{\mathfrak t}_3\hat\otimes\mathbf k)$. 

\begin{theorem}\label{thm:4} (see Cor. \ref{thm:gp:char})
One has 
\begin{equation}%\label{main:int}
\mathsf{GRT}_1(\mathbf k)=\mathrm{Stab}_{\mathrm{exp}(\hat{\mathfrak G}_\inert\hat\otimes\mathbf k)\times 
\mathrm{Out}_1(\hat{\mathfrak t}_4\hat\otimes\mathbf k)}(\overline\mu_{123},\overline\mu_{124}) , 
\end{equation}
 the right-hand side being the stabilizer group relative to the diagonal of action of the subgroup 
$\mathrm{exp}(\hat{\mathfrak G}_\inert\hat\otimes\mathbf k)\times 
\mathrm{Out}_1(\hat{\mathfrak t}_4\hat\otimes\mathbf k)$ of 
$\mathrm{Out}_1(\hat{\mathfrak t}_4\hat\otimes\mathbf k)\times 
\mathrm{Out}_1(\hat{\mathfrak t}_4\hat\otimes\mathbf k)$ on the Cartesian square of 
$\mathrm{OutHom}(\mathfrak t_3,\mathfrak t_4)(\mathbf k)$.
\end{theorem}
Thm. \ref{thm:3} leads to the analogous statement (Cor. \ref{thm:gp:char:sphere}), the Lie algebras 
$\mathfrak t_n$ being replaced by 
$\mathfrak P_{\vec{n+1}}$. 

This paper is organized as follows. \S\ref{sect:1} is devoted to the proof of Thm. \ref{thm:main}; \S\ref{sec:2}
is devoted to the proof of Thm. \ref{thm:main:2}. In \S\ref{sec:3},  we study the stabilizer groups and Lie algebras 
associated to an outer morphism of Lie algebras. In \S\ref{sec:4}, we use this material to formulate and prove Thm. \ref{thm:4}.
Throughout the paper, $\mathbf k$ is a commutative $\mathbb Q$-algebra with unit.

% we show that the identifications of $\mathfrak{grt}_1$ 
% from Thm. \ref{thm:main} enable one to express its Lie algebra structure 
% in terms of stabilizers of Lie algebra morphisms, a notion introduced in
% Lem. \ref{lem:stab:las}. 

\bigskip 

{\bf Acknowledgements.} The work of H. Furusho was partially supported the following grants: 
%ANR project HighAGT ANR20-CE40-0016 (???), 
JSPS KAKENHI JP24K00520, JP24K21510. Both authors also thank 
Anton Alekseev, Nikita Markarian, Toyo Taniguchi and Muze Ren for discussions.

\section{Characterization of $\mathfrak{grt}_1$ via Lie algebras of infinitesimal braids on the plane}\label{sect:1}

This section is mainly devoted to the proof of Thm. \ref{thm:main}. We first introduce its ingredients: the family of infinitesimal 
braid Lie algebras $(\mathfrak t_n)_{n\geq2}$ and the morphisms relating them (\S\ref{sect:1:1}); the Lie algebra 
$\mathfrak{grt}_1$ (\S\ref{sect:12}). \S\ref{sect:aux} contains some preliminary results to \S\ref{sect:1:4}, where 
we formulate and prove Thm. \ref{thm:main}. A variant of this result is proven in \S\ref{sect:variant}. 

\subsection{Lie algebras of infinitesimal braids on the plane}\label{sect:1:1}

For $n\geq2$, let $\mathfrak t_n$ be the Lie algebra with generators $t_{ij}$ with $i\neq j\in[1,n]$, 
subject to relations $t_{ji}=t_{ij}$ for $i\neq j$, $[t_{ij}+t_{ik},t_{jk}]=0$ for any 
distinct $i,j,k$, and $[t_{ij},t_{kl}]=0$ for any distinct $i,j,k,l$. Then $\mathfrak t_n$
is positively graded, each $t_{ij}$ being given degree 1. 

A partially defined map $\phi : [1,m]\supset S_\phi\to [1,n]$ may be identified with the sequence 
$\phi^{-1}(1),...,\phi^{-1}(n)$ of its fibers. Such a partially defined map defines a Lie algebra 
morphism $\mathfrak t_n\to \mathfrak t_m$, $x\mapsto x^\phi$, by 
$t_{ij}\mapsto \sum_{i',j':\phi(i')=i,\phi(j')=j}t_{i'j'}$; it is homogeneous of degree 0. 

Recall that the center $Z(\mathfrak t_n)$ of $\mathfrak t_n$ is equal to 
$\mathbf kz_n$, where $z_n:=\sum_{1\leq i<j\leq n}t_{ij}$. Define 
$\mathfrak P_{n+1}:=\mathfrak t_n/Z(\mathfrak t_n)$ and let $\pi_n : \mathfrak t_n\to\mathfrak P_{n+1}$
be the canonical projection. For $i\neq j\in[1,n+1]$, 
denote by $X_{ij}\in\mathfrak P_n$ the elements such that $X_{ij}:=\pi_n(t_{ij})$ if 
$i\neq j\in[1,n]$ and $X_{i,n+1}=X_{n+1,i}=-\sum_{j:j\neq i,j\in [1,n]}X_{ij}$
for $i\in[1,n]$. 

\subsection{The Lie algebra $\mathfrak{grt}_1$}\label{sect:12}

Let $\mathfrak{lie}(e_0,e_1)$ be the free Lie algebra over generators $e_0,e_1$; we set $e_\infty:=-e_0-e_1$; this Lie algebra is graded, 
with $e_0,e_1$ being given degree 1. 
Let $\mathfrak G\subset \mathfrak{lie}(e_0,e_1)$ be 
the sum of components of degree $>1$. 

Define a linear map 
$$
\mathrm{pent} : \mathfrak t_3\to \mathfrak t_4, C\mapsto C^{2,3,4}+C^{1,23,4}+C^{1,2,3}-C^{12,3,4}-C^{1,2,34}. 
$$
The Lie algebra $\mathfrak{lie}(e_0,e_1)$ is injected in 
$\mathfrak t_3$ by $e_0\mapsto t_{23}$, $e_1\mapsto t_{12}$. 

\begin{definition} (see \cite{Dr})
$\mathfrak{grt}_1$ is the subspace of $\mathfrak{lie}(e_0,e_1)$ of all elements $C(e_0,e_1)$ such that 
$$
C(e_0,e_1)+C(e_1,e_0)=C(e_0,e_1)+C(e_1,e_\infty)+C(e_\infty,e_0)=[e_1,C(e_0,e_1)]+[e_\infty,C(e_0,e_\infty)]=0
$$
and $\mathrm{pent}(C)=0$. 
\end{definition}
Define a Lie algebra structure on $\mathfrak G$ by $\langle C,C'\rangle:=-[C,C']-d_{C'}(C)+d_C(C')$, where 
$d_C$ is the derivation given by $e_0\mapsto[C,e_0],e_1\mapsto 0$. 

\begin{lemma}\label{lem:dr} (see \cite{Dr})
$\mathfrak{grt}_1$ is a Lie subalgebra of $(\mathfrak G,\langle,\rangle)$. 
\end{lemma}

The restriction of $\mathrm{pent}$ to 
$\mathfrak G \subset \mathfrak{lie}(e_0,e_1)$ is denoted $\mathrm{pent}\vert_{\mathfrak G}$. 

\begin{theorem}[see \cite{F1}]\label{thm:hf}
One has $\mathrm{ker}(\mathrm{pent}\vert_{\mathfrak G})=\mathfrak{grt}_1+\mathbf k[e_0,e_1]$, where 
$\mathfrak{grt}_1\subset\mathfrak{lie}(e_0,e_1)$ is the subspace defined as in \cite{Dr}, p. 851 (with $A:=e_0,B:=e_1$). 
\end{theorem}

\begin{proof}
Let $\mathfrak G_0\subset\mathfrak G$ be the set of elements $C$ satisfying the 2-cycle relation
$C(e_0,e_1)+C(e_1,e_0)=0$. 

Let $\mathrm{ihpent} : \mathfrak G\to \mathfrak P_5$ be the linear map $C\mapsto 
C(X_{12},X_{23})+C(X_{23},X_{34})+C(X_{34},X_{45})+C(X_{45},X_{51})+C(X_{51},X_{12})$. 
The composition of $\pi\circ(x\mapsto x^{4,3,2,1})\circ \mathrm{pent}\vert_{\mathfrak G}$
is the map $\mathfrak G\to \mathfrak P_5$ given by $C\mapsto 
C(X_{12},X_{23})+C(X_{45},X_{51})+C(X_{23},X_{34})-C(X_{21},X_{15})-C(X_{54},X_{43})$, 
by virtue of the equalities\footnote{We set $X_{a,bc}=X_{bc,a}:=X_{ab}+X_{ab}$ 
for distinct $a,b,c$} $X_{a,bc}+X_{bc}=X_{de}$ for $a,b,c,d,e$ distinct in $[1,5]$, of the commutativity relations
$[X_{a,bc},X_{bc}]=0$ for $a,b,c$ distinct in $[1,5]$, and of the fact that $\mathfrak G$ is supported in degree $>1$. 

The composition of both $\mathrm{ihpent}$ and 
$\pi\circ(x\mapsto x^{4,3,2,1})\circ \mathrm{pent}\vert_{\mathfrak G}$ with the Lie algebra
morphism $\mathfrak P_5\to\mathfrak P_4$, $X_{ij}\mapsto X_{ij}$ if $i,j\neq5$ and $X_{i5}=X_{5i}\mapsto 0$ if $i\neq 5$
is the map $\mathfrak G\to\mathfrak P_4$, $C\mapsto C(X_{12},X_{23})+C(X_{23},X_{34})=C(X_{12},X_{23})+C(X_{32},X_{21})$, 
where the equality follows from the relation $X_{12}=X_{34}$ in $\mathfrak P_4$. This is the composition of the self-map
$C(e_0,e_1)\mapsto C(e_0,e_1)+C(e_1,e_0)$ of $\mathfrak G$ with the composition 
$\mathfrak G\subset \mathfrak{lie}(e_0,e_1)\stackrel\sim{\to}\mathfrak P_4$, where the isomorphism is induced by 
$e_0\mapsto X_{12}$, $e_1\mapsto X_{23}$. It follows that the kernels of both $\mathrm{ihpent}$ and
 $\pi\circ(x\mapsto x^{4,3,2,1})\circ \mathrm{pent}\vert_{\mathfrak G}$ are contained in $\mathfrak G_0$. 
 
 The restriction to $\mathfrak G_0$ of the maps $\pi\circ(x\mapsto x^{4,3,2,1})\circ \mathrm{pent}\vert_{\mathfrak G}$
 and $\mathrm{ihpent}$ coincide; this implies the equality of $\mathrm{ker}(\mathrm{ihpent})$ and 
 $\mathrm{ker}(\pi\circ(x\mapsto x^{4,3,2,1})\circ \mathrm{pent}\vert_{\mathfrak G})$. 
  Since $\mathfrak G$ is supported in degree $>1$ and since $\pi$ is an isomorphism in these degrees, and since 
  $(x\mapsto x^{4,3,2,1})$ is an automorphism, the latter kernel is equal to $\mathrm{ker}(\mathrm{pent}\vert_{\mathfrak G})$, hence 
  $  \mathrm{ker}(\mathrm{pent}\vert_{\mathfrak G})=\mathrm{ker}(\mathrm{ihpent})$. 
 The result then follows from the combination of this equality with the equality 
 $\mathrm{ker}(\mathrm{ihpent})=\mathfrak{grt}_1+\mathbf k[e_0,e_1]$ from   
  \cite{F1}, Thm. 3. 
 \end{proof}

Define $\mathfrak G_{\inert}$ to be the set of $C\in \mathfrak G$ such that there exists 
$h_C\in \mathfrak G$ (necessarily unique) such that $[C,e_0]+[h_C,e_\infty]=0$. 

\begin{lemma}\label{lem:def:DC}
There is a map $\mathfrak G_{\inert}\to \mathfrak{der}(\mathfrak t_3)$, $C\mapsto D_C$, such that
$D_C$ is given by $t_{12}\mapsto 0$, $t_{23}\mapsto [C(t_{23},t_{12}),t_{23}]$, $t_{13}\mapsto [h_C(t_{23},t_{12}),t_{13}]$; $(\mathfrak G_{\inert},\langle,\rangle)$ is a Lie subalgebra of 
$(\mathfrak G,\langle,\rangle)$ and $C\mapsto D_C$ is a Lie algebra morphism. 
\end{lemma}

\begin{proof}
The only relation in $\mathfrak t_3$ is the centrality of $t_{12}+t_{13}+t_{23}$, and this element is taken to 0 by 
$D_C$ as can be seen by applying to identity $[C(e_0,e_1),e_0]+[h_C(e_0,e_1),e_\infty]=0$ the morphism 
$\mathfrak{lie}(e_0,e_1)\to\mathfrak t_3$, $e_0\mapsto t_{32},e_1\mapsto t_{21}$ and by the centrality of $t_{12}+t_{13}+t_{23}$. The fact that $(\mathfrak G_{\inert},\langle,\rangle)$ is a Lie subalgebra of 
$(\mathfrak G,\langle,\rangle)$ follows from \cite{EF}, Lem. 0.7(c), 
and the last statement is a direct verification. 
\end{proof}

\subsection{Auxiliary results}\label{sect:aux}

For $\mathfrak a$ a Lie algebra and $a\in \mathfrak a$, we set $\mathrm C_a(\mathfrak a):=\{x\in\mathfrak a|[a,x]=0\}$. 

\begin{lemma}\label{lem:Ct12}
For any $n\geq2$ and any $i,j\in[1,n]$ with $i\neq j$, 
one has $\mathrm C_{\mathfrak t_n}(t_{ij})=\mathbf k t_{ij}+\mathfrak t_{n-1}^{\phi}$, where 
$\phi :[1,n]\to [1,n-1]$ is any surjective map such that the fiber of cardinality $>1$ is $\{i,j\}$
(if $i<j$, then $\phi$ may be the map $\alpha\mapsto \alpha$ for $\alpha<j$, $\alpha\mapsto \alpha-1$ for 
$\alpha>j$, and $j\mapsto i$). 
\end{lemma}

\begin{proof}
One can assume $i=1,j=2$, in which case the result to be proved is 
$\mathrm C_{\mathfrak t_n}(t_{12})=\mathbf k t_{12}+\mathfrak t_{n-1}^{12,3,\ldots,n}$. 
Denoting by $\langle S\rangle$ the Lie subalgebra
generated by a family $S$, one has 
\begin{align}\label{ihara:explicit}
&\nonumber \mathrm C_{\mathfrak P_{n+1}}(X_{12})
\textstyle =\langle X_{12},X_{ij} : 3\leq i\neq j\leq n+1\rangle
=\mathbf kX_{12}
+\langle X_{ij} : 3\leq i\neq j\leq n\rangle
\\ &\nonumber \textstyle
=\mathbf kX_{12}
+\langle \{X_{ij} : 3\leq i<j\leq n+1\}\cup\{-\sum_{j:j\neq i,j\in [1,n]}X_{ij} : 3\leq i\leq n\}\rangle
\\ & \textstyle
=\mathbf kX_{12}
+\langle \{X_{ij} : 3\leq i<j\leq n\}\cup\{X_{1i}+X_{2i} : 3\leq i\leq n\}\rangle
\end{align}
where the first equality follows from 
\cite{I1}, Prop. 3.3.1, 
the second equality follows from the relation $[X_{12},X_{ij}]=0$ for any $i,j$ with 
$3\leq i\neq j\leq n+1$, the third equality follows from  the relation $\sum_{i=1}^nX_{i,n+1}=0$, and the 
last equality follows from the fact that the terms $X_{ij},i\geq3 $  in $\sum_{j:j\neq i,j\in [1,n]}X_{ij}$
(where $3\leq i\leq n$) belong to the generating set $\{X_{ij} : 3\leq i<j\leq n+1\}$. 

The preimage $\pi^{-1}(\mathrm C_{\mathfrak P_{n+1}}(X_{12}))$ of $\mathrm C_{\mathfrak P_{n+1}}(X_{12})$ 
by $\pi$ is $\{x\in \mathfrak t_n|[t_{12},x]\in 
Z(\mathfrak t_n)\}$. Since for any $x\in\mathfrak t_n$, $[t_{12},x]$ is a sum of terms of degrees $>1$, and since 
$Z(\mathfrak t_n)$  is contained in the degree $1$ part of $\mathfrak t_n$, the relation $[t_{12},x]\in 
Z(\mathfrak t_n)$ for $x\in \mathfrak t_n$ implies $[t_{12},x]=0$. It follows  
$$
\pi^{-1}(\mathrm C_{\mathfrak P_{n+1}}(X_{12}))=\mathrm C_{\mathfrak t_n}(t_{12}).
$$
On the other hand \eqref{ihara:explicit} implies
$$
\pi^{-1}(\mathrm C_{\mathfrak P_{n+1}}(X_{12}))
=\pi^{-1}(\mathbf kX_{12}
+\langle \{X_{ij} : 3\leq i<j\leq n\}\cup\{X_{1i}+X_{2i} : 3\leq i\leq n\}\rangle). 
$$
The combination of the two previous equalities implies 
\begin{equation}\label{partial:comm}
\mathrm C_{\mathfrak t_n}(t_{12})=\pi^{-1}(\mathbf kX_{12}
+\langle \{X_{ij} : 3\leq i<j\leq n\}\cup\{X_{1i}+X_{2i} : 3\leq i\leq n\}\rangle). 
\end{equation}

Recall that the Lie algebra $\mathfrak t_{n-1}$ is generated by the set
$\{t_{ij} : 1\leq i<j\leq n-1\}$. It follows that its image 
$\mathfrak t_{n-1}^{12,3,\ldots,n}$ by the morphism 
$\mathfrak t_{n-1}\to\mathfrak t_n$, $x\mapsto x^{12,3,\ldots,n}$ is 
the Lie subalgebra of $\mathfrak t_n$ generated by the image of this set, which is 
$\{t_{ij} : 3\leq i<j\leq n\}\cup\{t_{1i}+t_{2i} : 3\leq i\leq n\}$. 
The image of $\mathfrak t_{n-1}^{12,3,\ldots,n}$ by the morphism $\pi$
is therefore the Lie subalgebra of $\mathfrak P_{n+1}$ generated by 
$\pi(\{t_{ij} : 3\leq i<j\leq n\}\cup\{t_{1i}+t_{2i} : 3\leq i\leq n\})$, which is 
$\{X_{ij} : 3\leq i<j\leq n\}\cup\{X_{1i}+X_{2i} : 3\leq i\leq n\}$. Therefore 
$$
\pi(\mathfrak t_{n-1}^{12,3,\ldots,n})=\langle 
\{X_{ij} : 3\leq i<j\leq n\}\cup\{X_{1i}+X_{2i} : 3\leq i\leq n\}
\rangle.  
$$
Since $\pi(t_{12})=X_{12}$, one has $\pi(\mathbf kt_{12})=\mathbf kX_{12}$, which together with the previous equality 
implies 
$$
\pi(\mathbf kt_{12}+\mathfrak t_{n-1}^{12,3,\ldots,n})=\mathbf kX_{12}+\langle 
\{X_{ij} : 3\leq i<j\leq n\}\cup\{X_{1i}+X_{2i} : 3\leq i\leq n\}
\rangle.  
$$
As the kernel of $\pi$ is $Z(\mathfrak t_n)$, this implies  
$$
\pi^{-1}(\mathbf kX_{12}+\langle 
\{X_{ij} : 3\leq i<j\leq n\}\cup\{X_{1i}+X_{2i} : 3\leq i\leq n\})
=\mathbf kt_{12}+\mathfrak t_{n-1}^{12,3,\ldots,n}+Z(\mathfrak t_n). 
$$
Together with \eqref{partial:comm}, this implies 
\begin{equation}\label{quasi:final:comm}
\mathrm C_{\mathfrak t_n}(t_{12})=\mathbf kt_{12}+\mathfrak t_{n-1}^{12,3,\ldots,n}+Z(\mathfrak t_n). 
\end{equation}
The equality $z_n=t_{12}+z_{n-1}^{12,3,\ldots,n}$ implies 
$Z(\mathfrak t_n)\subset \mathbf kt_{12}+\mathfrak t_{n-1}^{12,3,\ldots,n}$, therefore 
$\mathbf kt_{12}+\mathfrak t_{n-1}^{12,3,\ldots,n}+Z(\mathfrak t_n)=\mathbf kt_{12}+\mathfrak t_{n-1}^{12,3,\ldots,n}$. 
The result follows from the combination of this equality with \eqref{quasi:final:comm}. 
\end{proof}

\begin{lemma}\label{lem:kertt}
The kernel of the linear endomorphism of $\mathfrak t_4$ given by $x\mapsto [t_{14},[t_{23},x]]$ is 
$\mathrm{ker}(x\mapsto [t_{14},[t_{23},x]])=\mathfrak t_3^{1,23,4}+\mathfrak t_3^{14,2,3}$. 
\end{lemma}

\begin{proof}
It follows from the relations $[t_{23},\mathfrak t_3^{1,23,4}]=[t_{14},\mathfrak t_3^{14,2,3}]=0$ and $[t_{14},t_{23}]=0$
that $\mathrm{ker}(x\mapsto [t_{14},[t_{23},x]])$ contains $\mathfrak t_3^{1,23,4}+\mathfrak t_3^{14,2,3}$; let us now prove 
the opposite inclusion, i.e. 
\begin{equation}\label{inclusion:ker:tt}
\mathrm{ker}(x\mapsto [t_{14},[t_{23},x]])\subset 
\mathfrak t_3^{1,23,4}+\mathfrak t_3^{14,2,3}. 
\end{equation}
 Let $x\in \mathrm{ker}(x\mapsto [t_{14},[t_{23},x]])$. Then 
$[t_{14},[t_{23},x]]=0$. By Lem. \ref{lem:Ct12}, and since $x$ belongs to the part of $\mathfrak t_3$ of degree $>1$, 
this implies the existence of $y\in \mathfrak t_3$, 
such that 
\begin{equation}\label{temp:eq}
    [t_{23},x]=y^{14,2,3}. 
\end{equation}
Applying $x\mapsto x^{1,2,3,\emptyset}$ to this equality, one obtains $y=[t_{23},x^{1,2,3,\emptyset}]$, and further applying 
$x\mapsto x^{14,2,3}$, one obtains $y^{14,2,3}=[t_{23},x^{14,2,3,\emptyset}]$. Combining the
latter equality with \eqref{temp:eq}, one obtains $[t_{23},x-x^{14,2,3,\emptyset}]=0$, which by Lem. \ref{lem:Ct12}
implies the existence of $z\in\mathfrak t_3$ and $z_0\in\mathbf k$ such that $x-x^{14,2,3,\emptyset}=z^{1,23,4}+z_0t_{23}$. 
Hence $x=z^{1,23,4}+x^{14,2,3,\emptyset}+z_0t_{23}$. 
If one sets $w:=x^{1,2,3,\emptyset}+z_0t_{23}$, one then obtains  
$$
x=z^{1,23,4}+w^{14,2,3},  
$$
where $w,z\in \mathfrak t_3$. 
Therefore $x\in \mathfrak t_3^{1,23,4}+\mathfrak t_3^{14,2,3}$. This implies \eqref{inclusion:ker:tt}.
\end{proof}

\begin{lemma}\label{lem:act:grt}
(a) One has the inclusion $\mathfrak{grt}_1\subset\mathfrak G_{\inert}$. 

(b) For any $C\in\mathfrak{grt}_1$, there is a unique derivation $D$ of $\mathfrak t_4$ such that\footnote{We use the shorthand 
notation $t_{i_1...i_s,j}=t_{j,i_1...i_s}:=t_{i_1j}+...+t_{i_sj}$ for any distinct $i_1,...,i_s,j$.} 
\begin{align}\label{formulas:D}
& D : t_{12}\mapsto 0,\quad t_{23}\mapsto[C(t_{32},t_{21}),t_{23}],\quad t_{13}\mapsto [C(t_{31},t_{12}),t_{13}], \quad 
\\ & \nonumber
t_{14}\mapsto [C(t_{41},t_{12})+C(t_{4,12},t_{12,3}),t_{14}],\, 
t_{24}\mapsto [C(t_{42},t_{21})+C(t_{4,21},t_{21,3}),t_{24}],\, t_{34}\mapsto [C(t_{43},t_{3,12}),t_{34}]. 
\end{align}
\end{lemma}

\begin{proof}
(a) Let $C\in\mathfrak{grt}_1$. By relation (5.19) in \cite{Dr}, $C$ satisfies $[e_0,C(e_1,e_0)]+[e_\infty,C(e_1,e_\infty)]=0$. 
By relation (5.19) in \cite{Dr}, $C$ satisfies the 2-cycle relation $C(y,x)=-C(x,y)$, which implies 
both $C(e_1,e_0)=-C(e_0,e_1)$ and $C(e_1,e_\infty)=-C(e_0,e_\infty)$. One then obtains
\begin{equation}\label{inert:C}
    [e_0,C(e_0,e_1)]+[e_\infty,C(e_\infty,e_1)]=0, 
\end{equation} 
therefore $C\in \mathfrak G_{\inert}$ (with $h_C=C(e_\infty,e_1)$). 

(b) Observe that the injection $\mathfrak{lie}(e_0,e_1)\to \mathfrak t_3$ takes
$C(e_\infty,e_1)$ to $C(-t_{2,13},t_{12})$, which is equal to $C(t_{31},t_{12})$ as $C$ is a sum of terms of degrees $>1$
and by the centrality of $t_{2,13}+t_{13}$; taking the image of \eqref{inert:C} by this injection and using again the centrality of 
$t_{2,13}+t_{13}$, one obtains  
\begin{equation}\label{inert:C:in:t3}
    [t_{23},C(t_{32},t_{21})]+[t_{13},C(t_{31},t_{12})]=0 
\end{equation} 
(equality in $\mathfrak t_3$).

One can show that a presentation of $\mathfrak t_4$ is as follows: (i) $\sum_{i,j : 1\leq i<j\leq 4}t_{ij}$ is central, 
(ii) the sum $t_{12}+t_{13}+t_{23}$ commutes with each of its summands, (iii) the sum $t_{12}+t_{14}+t_{24}$ 
commutes with each of its summands, (iv) $[t_{13},t_{24}]=0$, (v) $[t_{14},t_{23}]=0$. 

Let us denote by $D(t_{ij})$ the images of $t_{ij}$ indicated in the statement. 

One has $D(t_{12})+D(t_{13})+D(t_{23})=[C(t_{32},t_{21}),t_{23}]+[C(t_{31},t_{12}),t_{13}]=0$ by virtue of 
the image of \eqref{inert:C:in:t3} by the morphism $\mathfrak t_3\to \mathfrak t_4$, $x\mapsto x^{1,2,3}$, therefore
\begin{equation}\label{sum:im:D}
D(t_{12})+D(t_{13})+D(t_{23})=0.     
\end{equation} 
Then for $i\neq j\in\{1,2,3\}$, one has 
$$
[t_{12}+t_{13}+t_{23},D(t_{ij})]+[D(t_{12})+D(t_{13})+D(t_{23}),t_{ij}]=0,  
$$
where the first term vanishes since $D(t_{ij})$ belongs to $\langle t_{12},t_{13},t_{23}\rangle$ and 
$t_{12}+t_{13}+t_{23}$ is central in this Lie algebra, and the second term vanishes by \eqref{sum:im:D}.
This implies that (ii) is preserved. 

One has 
\begin{align}\label{im:D:t124}
& D(t_{12})+D(t_{14})+D(t_{24})=[C(t_{41},t_{12})+C(t_{4,12},t_{12,3}),t_{14}]+[C(t_{42},t_{21})+C(t_{4,21},t_{21,3}),t_{24}]
\\ \nonumber
& =[C(t_{41},t_{12}),t_{14}]+[C(t_{42},t_{21}),t_{24}]+[C(t_{4,12},t_{12,3}),t_{12,4}]
=[C(t_{4,12},t_{12,3}),t_{12}+t_{14}+t_{24}], 
\end{align}
where the second equality follows from $t_{4,21}=t_{4,12}$ and 
$t_{21,3}=t_{12,3}$, and the third equality follows from the commutation of $t_{12}$ with 
$t_{4,12}$ and $t_{12,3}$ and from the image of \eqref{inert:C:in:t3} by the morphism 
$\mathfrak t_3\to \mathfrak t_4$, given by $x\mapsto x^{1,2,4}$. Then 
\begin{align*}
&[D(t_{12})+D(t_{14})+D(t_{24}),t_{ij}]+[t_{12}+t_{14}+t_{24},D(t_{ij})]
\\& =[[C(t_{4,12},t_{12,3}),t_{12}+t_{14}+t_{24}],t_{ij}]+[t_{12}+t_{14}+t_{24},D(t_{ij})]
\\ & =[[C(t_{4,12},t_{12,3}),t_{ij}],t_{12}+t_{14}+t_{24}]+[t_{12}+t_{14}+t_{24},D(t_{ij})]
\\ & = [t_{12}+t_{14}+t_{24},D(t_{ij})-[C(t_{4,12},t_{12,3}),t_{ij}]]=0,  
\end{align*}
where the first equality follows from \eqref{im:D:t124}, the second equality follows from the 
centrality of $t_{12}+t_{14}+t_{24}$ in $\langle t_{12},t_{24},t_{14}\rangle$, the third equality follows 
from the conjunction of the fact that for any $i\neq j\in\{1,2,4\}$, 
$D(t_{ij})-[C(t_{4,12},t_{12,3}),t_{ij}]\in\langle t_{12},t_{24},t_{14}\rangle$ and of the 
centrality of $t_{12}+t_{14}+t_{24}$ in $\langle t_{12},t_{24},t_{14}\rangle$.  
It follows that (iii) is preserved. 

One computes
\begin{align*}
    &\textstyle\sum_{i,j : 1\leq i<j\leq 4}D(t_{ij})
    =D(t_{12})+D(t_{14})+D(t_{24})+D(t_{34})
    \\ & =[C(t_{4,12},t_{12,3}),t_{12}+t_{14}+t_{24}]+[C(t_{43},t_{3,12}),t_{34}]
    \\ & =[C(t_{4,12},t_{12,3}),t_{12,4}]+[C(t_{43},t_{3,12}),t_{34}] =0
\end{align*}
where the first equality follows from \eqref{sum:im:D} and $D(t_{12})=0$, the second equality follows from 
\eqref{im:D:t124}, the third equality follows from the commutation of $t_{12}$ with 
$t_{4,12}$ and $t_{12,3}$, and the last equality follows from the image of \eqref{inert:C:in:t3} by the morphism
$\mathfrak t_3\to \mathfrak t_4$ given by $x\mapsto x^{3,12,4}$.
Therefore 
\begin{equation}\label{van:im:center}
\textstyle    \sum_{i,j : 1\leq i<j\leq 4}D(t_{ij})=0, 
\end{equation} 
which implies that (i) is preserved. 

One computes 
\begin{align*}
    & [D(t_{13}),t_{24}]+[t_{13},D(t_{24})]
=[[C(t_{31},t_{12}),t_{13}],t_{24}]+[t_{13},[C(t_{42},t_{21})+C(t_{4,21},t_{21,3}),t_{24}]]
\\ & =[[C(t_{31},t_{12})-C(t_{42},t_{21})-C(t_{4,21},t_{21,3}),t_{13}],t_{24}]
=[[C(t_{31},t_{12})+C(t_{12},t_{24})+C(t_{3,12},t_{12,4}),t_{13}],t_{24}]
\\ & =[[C(t_{31,2},t_{24})+C(t_{31},t_{1,24}),t_{13}],t_{24}]=
[[C(t_{31,2},t_{24}),t_{13}],t_{24}]+[[C(t_{31},t_{1,24}),t_{24}],t_{13}]=0
\end{align*}
where the second and fifth equalities follow from the commutation relation $[t_{13},t_{24}]=0$, the third equality
follows from $C(y,x)=-C(x,y)$, which is the 2-cycle equation (5.17) in \cite{Dr}, 
the fourth equality follows from the pentagon relation (5.20) in \cite{Dr}, 
and the last equality follows from the relations $[t_{31,2},t_{13}]=[t_{24},t_{13}]=0$ and 
$[t_{31},t_{24}]=[t_{1,24},t_{24}]=0$.  Therefore $[D(t_{13}),t_{24}]+[t_{13},D(t_{24})]=0$, which implies that (iv) is preserved. 

Similarly, one computes 
\begin{align*}
    & [D(t_{14}),t_{23}]+[t_{14},D(t_{23})]
=[[C(t_{41},t_{12})+C(t_{4,12},t_{12,3}),t_{14}],t_{23}]+[t_{14},[C(t_{32},t_{21}),t_{23}]]
\\ & =[[C(t_{41},t_{12})+C(t_{4,12},t_{12,3})-C(t_{32},t_{21}),t_{14}],t_{23}]
=[[C(t_{41},t_{12})+C(t_{4,12},t_{12,3})+C(t_{12},t_{23}),t_{14}],t_{23}]
\\ & =[[C(t_{41,2},t_{23})+C(t_{41},t_{1,23}),t_{14}],t_{23}]=
[[C(t_{41,2},t_{23}),t_{14}],t_{23}]+[[C(t_{41},t_{1,23}),t_{23}],t_{14}]=0
\end{align*}
where the second and fifth equalities follow from the commutation relation $[t_{23},t_{14}]=0$, the third equality
follows from the 2-cycle identity $C(y,x)=-C(x,y)$, the fourth equality follows from the pentagon relation, 
and the last equality follows from the relations $[t_{41,2},t_{14}]=[t_{23},t_{14}]=0$ and 
$[t_{41},t_{23}]=[t_{1,23},t_{23}]=0$.  Therefore $[D(t_{14}),t_{23}]+[t_{14},D(t_{23})]=0$, which implies that (v) is preserved. 

Since the assignment $t_{ij}\mapsto D(t_{ij})$ preserves all the relations of $\mathfrak t_4$, it induces a derivation of
this Lie algebra. 
\end{proof}

\begin{remark}
One can check the action of Lem. \ref{lem:act:grt} to coincide with the one arising from 
the action of $\mathfrak{grt}_1$ on $\mathbf{PaCD}$ in \cite{BN}, corresponding to the object
$((\bullet\bullet)\bullet)\bullet$. 
\end{remark}

\subsection{Characterization of $\mathfrak{grt}_1$ in terms of the Lie algebras $\mathfrak{t}_3$ and $\mathfrak{t}_4$}
\label{sect:1:4}

\begin{theorem}\label{thm:main}
The subset $\mathfrak{grt}_1\subset\mathfrak G$ is equal to the subset of $\mathfrak G_\inert$ of all elements  
$C$ such that there exists $D\in\mathfrak{der}(\mathfrak t_4)$ and $X\in\mathfrak t_4$, 
satisfying 

(a) the morphism $\mu_{123}:\mathfrak t_3\to\mathfrak t_4,x\mapsto x^{1,2,3}$ intertwines $D_C$ and $D$
(see Lem. \ref{lem:def:DC}), and 

(b) the morphism $\mu_{124}:\mathfrak t_3\to\mathfrak t_4,x\mapsto x^{1,2,4}$ intertwines $D_C$ and $D+\mathrm{ad}_X$. 
\end{theorem}

It will be shown in the course of the proof that for any $(C,D,X)$ as in this statement, one has 
$X \in \mathfrak t_3^{12,3,4}\subset \mathfrak t_4$. 

\begin{proof}
Let $\mathfrak X$ be the subset of $\mathfrak G_\inert$ defined in the statement. 

Let us show the inclusion 
\begin{equation}\label{incl:grt:i}
\mathfrak{grt}_1\subset\mathfrak X. 
\end{equation}
Let $C\in \mathfrak{grt}_1$. By Lem. \ref{lem:act:grt}(a), $C\in\mathfrak G_\inert$. Let $D$ be the derivation of 
$\mathfrak t_4$ induced by $C$ as in Lem. \ref{lem:act:grt}(b), and let $X:=C(t_{4,12},t_{12,3})$. 
Let us check that $(D,X)$ satisfies conditions (a) and (b) from the statement. It follows from 
\eqref{inert:C:in:t3}, from Lem. \ref{lem:def:DC} and from the 2-cycle identity $C(y,x)=-C(x,y)$ 
that $D_C$ is given by 
\begin{equation}\label{assignt:DC}
t_{12}\mapsto 0, \quad t_{23}\mapsto [C(t_{32},t_{21}),t_{23}],\quad t_{13}\mapsto [C(t_{31},t_{12}),t_{13}]. 
\end{equation}
Let us check that condition (a) is satisfied. One has 
\begin{align*}
& D(t_{12}^{1,2,3})=D(t_{12})=0=D_C(t_{12})^{1,2,3}, 
\\& D(t_{23}^{1,2,3})=D(t_{23})=[C(t_{32},t_{21}),t_{23}]=[C(t_{32},t_{21}),t_{23}]^{1,2,3}=D_C(t_{23})^{1,2,3},
\\ & D(t_{13}^{1,2,3})=D(t_{13})=[C(t_{31},t_{12}),t_{13}]=[C(t_{31},t_{12}),t_{13}]^{1,2,3}=D_C(t_{13})^{1,2,3}
\end{align*}
by \eqref{assignt:DC} and \eqref{formulas:D}.
Both maps $x\mapsto D(x^{1,2,3})$ and $x\mapsto D_C(x)^{1,2,3}$ are derivations of $\mathfrak t_3$ with values in 
$\mathfrak t_4$ with respect to the Lie algebra morphism $x\mapsto x^{1,2,3}$. Since they coincide on a 
generating set of $\mathfrak t_3$, they are equal. 
This implies the identity $\forall x\in \mathfrak t_3,D(x^{1,2,3})=D_C(x)^{1,2,3}$, therefore 
the condition (a) is satisfied. 

On the other hand, since $X$ commutes with $t_{12}$, $D+\mathrm{ad}_X$ is such that 
\begin{align*}
& D+\mathrm{ad}_X : t_{12}\mapsto 0,\quad t_{24}\mapsto [C(t_{42},t_{21}),t_{24}], t_{14}\mapsto [C(t_{41},t_{12}),t_{14}], 
\\ & \nonumber
t_{23}\mapsto[C(t_{32},t_{21})+C(t_{3,12},t_{12,4}),t_{23}],\quad t_{13}\mapsto [C(t_{31},t_{12})+C(t_{3,12},t_{12,4}),t_{13}], 
\\ & \nonumber t_{34}\mapsto [C(t_{43},t_{3,12})+C(t_{3,12},t_{12,4}),t_{34}]. 
\end{align*}
Let us check that condition (b) is satisfied. One has 
\begin{align*}
& (D+\mathrm{ad}_X)(t_{12}^{1,2,4})=(D+\mathrm{ad}_X)(t_{12})=0=D_C(t_{12})^{1,2,4}, 
\\& (D+\mathrm{ad}_X)(t_{23}^{1,2,4})=(D+\mathrm{ad}_X)(t_{24})=[C(t_{42},t_{21}),t_{24}]
=[C(t_{32},t_{21}),t_{23}]^{1,2,4}=D_C(t_{23})^{1,2,4},
\\ & (D+\mathrm{ad}_X)(t_{13}^{1,2,4})=(D+\mathrm{ad}_X)(t_{14})=[C(t_{41},t_{12}),t_{14}]
=[C(t_{31},t_{12}),t_{13}]^{1,2,4}=D_C(t_{13})^{1,2,4}
\end{align*}
by \eqref{assignt:DC} and \eqref{formulas:D}.
Both maps $x\mapsto (D+\mathrm{ad}_X)(x^{1,2,4})$ and $x\mapsto D_C(x)^{1,2,4}$ are derivations of $\mathfrak t_3$ with values in 
$\mathfrak t_4$ with respect to the Lie algebra morphism $x\mapsto x^{1,2,4}$. Since they coincide on a 
generating set of $\mathfrak t_3$, they are equal. 
This implies the identity $\forall x\in \mathfrak t_3,(D+\mathrm{ad}_X)(x^{1,2,4})=D_C(x)^{1,2,4}$, therefore 
the condition (b) is satisfied. It follows that $C\in\mathfrak X$. This implies \eqref{incl:grt:i}. 

Let us now show the opposite inclusion, i.e. 
\begin{equation}\label{incl:i:ii}
\mathfrak X\subset \mathfrak{grt}_1. 
\end{equation}
Let $C\in\mathfrak X$. Let $D,X$ be such that conditions (a) and (b) are satisfied.  
The condition (a) is expressed as the identity $\forall x\in \mathfrak t_3,D(x^{1,2,3})=D_C(x)^{1,2,3}$, which 
implies the relations 
\begin{equation}\label{eq:1}
    D(t_{12})=0,D(t_{23})=[C(t_{32},t_{21}),t_{23}],D(t_{13})=[h_C(t_{32},t_{21}),t_{13}]. 
\end{equation}
On the other hand, the condition (b) is expressed as the identity $\forall x\in \mathfrak t_3,D(x^{1,2,4})+[X,x^{1,2,4}]=D_C(x)^{1,2,4}$, 
which implies the relations
\begin{equation}\label{eq:2}
 D(t_{12})+[X,t_{12}]=0,D(t_{24})+[X,t_{24}]=[C(t_{42},t_{21}),t_{24}],D(t_{14})+[X,t_{14}]=[h_C(t_{42},t_{21}),t_{14}]. 
\end{equation}
The combination of the first equalities in \eqref{eq:1} and \eqref{eq:2} implies 
\begin{equation}\label{eq:comm:X:t12}
[X,t_{12}]=0. 
\end{equation}
The two last equalities of \eqref{eq:2} imply 
\begin{equation}\label{eq:3}
    D(t_{24})=[C(t_{42},t_{21})-X,t_{24}],D(t_{14})=[h_C(t_{42},t_{21})-X,t_{14}]. 
\end{equation}
Recall that $[t_{14},t_{23}]=0$, which by the derivation property of $D$ implies $[D(t_{14}),t_{23}]+[t_{14},D(t_{23})]=0$. 
Combining this with \eqref{eq:1} and \eqref{eq:3}, one obtains 
\begin{equation}\label{eq:5}
    [[h_C(t_{42},t_{21}),t_{14}],t_{23}]+[t_{14},[C(t_{32},t_{21}),t_{23}]]=[[X,t_{14}],t_{23}]. 
\end{equation}
Similarly, the relation $[t_{13},t_{24}]=0$ implies $[D(t_{13}),t_{24}]+[t_{13},D(t_{24})]=0$, which
together with \eqref{eq:1} and \eqref{eq:3} yields 
\begin{equation*}
    [[h_C(t_{32},t_{21}),t_{13}],t_{24}]+[t_{13},[C(t_{42},t_{21}),t_{24}]]=[t_{13},[X,t_{24}]]. 
\end{equation*}
Applying the automorphism $x\mapsto x^{1,2,4,3}$ of $\mathfrak t_4$ to this equality, one obtains  
\begin{equation}\label{eq:6}
    [[h_C(t_{42},t_{21}),t_{14}],t_{23}]+[t_{14},[C(t_{32},t_{21}),t_{23}]]=[t_{14},[X^{1,2,4,3},t_{23}]]. 
\end{equation}
The difference of \eqref{eq:5} and \eqref{eq:6}, combined with the commutation relation $[t_{14},t_{23}]=0$, yields the equality
\begin{equation*}\label{eq:comm:XX:tt}
    [[X+X^{1,2,4,3},t_{14}],t_{23}]=0. 
\end{equation*}
The commutation relation $[t_{14},t_{23}]=0$, combined with \eqref{eq:5}, also 
yields the equality 
\begin{equation}\label{ChX:in:kertt}
[t_{14},[t_{23},C(t_{32},t_{21})-h_C(t_{42},t_{21})+X]]=0. 
\end{equation}
% By \eqref{eq:comm:X:t12}, one has $[t^{12},X+X^{1,2,4,3}]=0$, which together with \eqref{eq:comm:XX:tt}, 
% Lem. \ref{lem:int:Ct12:kertt} and the fact that $\mathfrak G$ lives in degree $>1$ implies 
% \begin{equation}\label{X+X43:belonging}
% X+X^{1,2,4,3}=0.
% \end{equation}
On the other hand, the combination of \eqref{eq:comm:X:t12} and Lem. \ref{lem:Ct12}, and 
the fact that $\mathfrak G$ lives in degree $>1$, implies the existence of 
$\lambda\in\mathfrak t_3$, such that 
\begin{equation}\label{eq:X}
X=\lambda^{12,3,4}.
\end{equation}
Combining \eqref{ChX:in:kertt} with Lem. \ref{lem:kertt}, one obtains the existence of $\gamma$ and 
$\delta$ in $\mathfrak t_3$, such that 
$$
C(t_{32},t_{21})-h_C(t_{42},t_{21})+X=\gamma^{23,1,4}+\delta^{14,2,3}, 
$$
which together with \eqref{eq:X} yields
\begin{equation}\label{main:equation:C:hC}
C(t_{32},t_{21})-h_C(t_{42},t_{21})
=\gamma^{23,1,4}+\delta^{14,2,3}-\lambda^{12,3,4}. 
\end{equation}
Since $C$ is a sum of terms of degree $>1$, the same is true of $h_C,\gamma$ and $\delta$; since on the other hand, 
$\mathfrak t_2$ is supported in degree 1, this implies that their images by any graded Lie algebra morphism 
from $\mathfrak{lie}(e_0,e_1)$ or $\mathfrak t_3$ to $\mathfrak t_2$ is zero. 

Applying to \eqref{main:equation:C:hC} the morphism $\mathfrak t_4\to\mathfrak t_3$, $x\mapsto x^{1,2,3,\emptyset}$, and using
$h_C(0,x)=\gamma^{1,2,\emptyset}=\delta^{1,2,\emptyset}=0$, one obtains
\begin{equation}\label{expr:delta:C}
\delta=C(t_{32},t_{21}). 
\end{equation}
Applying to \eqref{main:equation:C:hC} the morphism $\mathfrak t_4\to\mathfrak t_3$, $x\mapsto x^{\emptyset,1,2,3}$, and using 
$\gamma^{1,\emptyset,2}=0$, one obtains
$$
\lambda=\delta^{3,1,2},   
$$
which together with \eqref{expr:delta:C} yields
\begin{equation}\label{expr:lambda:C}
\lambda=C(t_{21},t_{13}). 
\end{equation}

Applying to \eqref{main:equation:C:hC} the morphism $\mathfrak t_4\to\mathfrak t_3$, $x\mapsto x^{2,1,\emptyset,3}$, and using $C(0,t_{12})=
\delta^{1,2,\emptyset}=\lambda^{1,2,\emptyset}=0$, one obtains
\begin{equation}\label{expr:gamma:hC}
\gamma=-h_C(t_{31},t_{12}). 
\end{equation}
Applying to \eqref{main:equation:C:hC} the morphism $\mathfrak t_4\to\mathfrak t_3$, $x\mapsto x^{2,\emptyset,1,3}$, and using 
$\delta^{1,\emptyset,2}=0$, one obtains
$$
\gamma=\lambda^{2,1,3} 
$$
which together with \eqref{expr:lambda:C} yields
\begin{equation}\label{expr:gamma:C}
\gamma=C(t_{12},t_{23}). 
\end{equation}
The combination of \eqref{expr:gamma:hC} and \eqref{expr:gamma:C} then gives, after applying $x\mapsto x^{2,3,1}$, the relation 
\begin{equation}\label{expr:hC:C}
    h_C(t_{12},t_{23})=-C(t_{23},t_{31}). 
\end{equation}
Plugging in relation \eqref{main:equation:C:hC} the relations \eqref{expr:hC:C}, \eqref{expr:gamma:C}, \eqref{expr:delta:C} and \eqref{expr:lambda:C}, 
one obtains 
$$
C(t_{32},t_{21})+C(t_{21},t_{14})
=C(t_{23,1},t_{14})+C(t_{32},t_{2,14})-C(t_{3,12},t_{12,4}). 
$$
which after applying the permutation automorphism $x\mapsto x^{2,3,4,1}$ of $\mathfrak t_4$ yields the pentagon relation
$$
C(t_{43},t_{32})+C(t_{32},t_{21})+C(t_{4,23},t_{23,1})
=C(t_{34,2},t_{21})+C(t_{43},t_{3,12}). 
$$
Therefore $C$ belongs to the kernel of $\mathrm{pent}\vert_{\mathfrak G}$. By Thm. \ref{thm:hf}, it follows that 
there exists $(c,c_0)\in\mathfrak{grt}_1\times\mathbf k$, such that $C=c+c_0[e_0,e_1]$. One has 
$C\in \mathfrak X\subset\mathfrak G_\inert$ (by the definition of $\mathfrak X$) and $c\in \mathfrak{grt}_1
\subset \mathfrak G_\inert$ (by Lem. \ref{lem:act:grt}(a)). Therefore $c_0[e_0,e_1]=C-c\in \mathfrak G_\inert$. 
Let then $h\in\mathfrak G$ be such that $c_0[[e_0,e_1],e_0]+[h,e_\infty]=0$. By the uniqueness of $h$, 
it is homogeneous of degree 2, therefore of the form $\alpha[e_0,e_1]$; then 
$c_0[[e_0,e_1],e_0]+\alpha[[e_0,e_1],e_\infty]=0$; since the degree 3 part of 
$\mathfrak G$ is a rank 2 free $\mathbf k$-module generated by  $[[e_0,e_1],e_0]$ and 
$[[e_0,e_1],e_\infty]=-[[e_0,e_1],e_0]-[[e_0,e_1],e_1]$, it follows that $\alpha=c_0=0$, hence 
$c_0=0$. Therefore $C=c\in\mathfrak{grt}_1$. This shows the inclusion \eqref{incl:i:ii}. 

The result follows from the combination of the inclusions \eqref{incl:i:ii} and \eqref{incl:grt:i}. 
\end{proof}

\subsection{Auxiliary results}

For $V$ a $\mathbb Z$-graded $\mathbf k$-module and $d\in\mathbb Z$, we define $V[d]$ to be the degree $d$ component of $V$, and for 
$S\subset \mathbb Z$, we define $V[S]$ to be the direct sum $\oplus_{d\in S}V[d]$. We write $V[>1]$ instead of $V[\{d|d>1\}]$. 

Denote by $Z(\mathfrak a)$ the center of a Lie algebra $\mathfrak a$. 

\begin{lemma}\label{lem:general}
Let $\pi :  \mathfrak a\to\mathfrak b$ be a surjective graded morphism of positively graded Lie algebras which is an 
isomorphism in degrees $>1$. Then: 

(a) $\mathrm{ker}\pi\subset Z(\mathfrak a)$; 

(b) $Z(\mathfrak a)$ is supported in degree $1$ if and only if so is $Z(\mathfrak b)$; 

(c) if the equivalent conditions of (b) are satisfied, then $Z(\mathfrak b)=\pi(Z(\mathfrak a))$. 
\end{lemma}

\begin{proof}
One has obviously $\pi^{-1}(Z(\mathfrak b))\supset Z(\mathfrak a)$.
Let $z\in Z(\mathfrak b)$ and $\tilde z\in \mathfrak a$ be a lift of
$z$. Then $[\tilde z,\mathfrak a]\subset \mathrm{ker}\pi\subset \mathfrak a[1]$; since $[\tilde z,\mathfrak a]$
is supported in degree $>1$, this implies $[\tilde z,\mathfrak a]=0$ hence $\tilde z\in Z(\mathfrak a)$. Hence 
$\pi^{-1}(Z(\mathfrak b))\subset Z(\mathfrak a)$. Therefore $\pi^{-1}(Z(\mathfrak b))=Z(\mathfrak a)$.
The statement follows. 
\end{proof}

\begin{lemma}\label{lem9}
If $\pi : \mathfrak a\to\mathfrak b$ be a morphism as in Lem. \ref{lem:general}, then there is a bijection 
$\{\text{graded sections of $\pi$}\}\to\{\text{sections of $\pi[1] : \mathfrak a[1]\to b[1]$}\}$, and 
any graded section of $\pi$ is a
morphism of Lie algebras. 
\end{lemma}

\begin{proof}
The said map takes a graded section $\sigma$ of $\pi$ to $\sigma[1]$, and its inverse takes a section 
of $\pi[1]$ to its direct sum with the inverse of the linear isomorphism $\pi[>1] : \mathfrak a[>1]\to\mathfrak b[>1]$. 

Let now $\sigma$ be a graded section of $\pi$. Let $b,b'$ be elements of $\mathfrak a$. Then 
$\pi([\sigma b,\sigma b'])=[\pi\sigma b,\pi\sigma b']=[b,b']$, therefore 
$[\sigma b,\sigma b']\equiv \sigma([b,b'])$ mod $\mathrm{ker}\pi$. Since $\mathrm{ker}\pi$ is supported in degree 1
and $[\mathfrak a,\mathfrak a]$ is supported in degree 1, this implies $[\sigma b,\sigma b']= \sigma([b,b'])$, therefore
$\sigma$ is a Lie algebra morphism. 
\end{proof}

Recall that the Lie algebra of derivations $\mathfrak{der}(\mathfrak a)$ of a $\mathbb Z$-graded 
Lie algebra $\mathfrak a$ is itself a $\mathbb Z$-graded Lie algebra. 

\begin{lemma}\label{lem:corresp}
Let $\pi : \mathfrak a\to\mathfrak b$ be a morphism as in Lem. \ref{lem:general}, satisfying the equivalent conditions of 
 Lem. \ref{lem:general}(b). Let $\sigma$ be any graded section of $\pi$. Then: 
 
 (a) the assignment $p : D\mapsto \pi D\sigma$ sets up a bijection $\mathfrak{der}(\mathfrak a)[>1]\to \mathfrak{der}(\mathfrak b)[>1]$, 
 whose inverse is given by $s : E\mapsto \sigma E\pi$; 
 
 (b) if $D\in \mathfrak{der}(\mathfrak a)[>1]$ and $E:=p(D)$, then  $\pi : \mathfrak a\to\mathfrak b$ 
 intertwines $D$ and $E$, and $\sigma : \mathfrak b\to\mathfrak a$ intertwines $E$ and $D$; 
 
 (c) if $a\in \mathfrak a[>1]$, then $p$  takes $\mathrm{ad}_a$ to $\mathrm{ad}_{\pi(a)}$, and 
 if $b\in\mathfrak b[>1]$, then $s$ takes $\mathrm{ad}_b$ to $\mathrm{ad}_{\sigma(b)}$. 
\end{lemma}

\begin{proof}
(a) If $D\in\mathfrak{der}(\mathfrak a)$, then $\pi D\sigma$ is a derivation of $\mathfrak b$ with values in the pull-back of the 
adjoint module $\mathfrak b$ by the endomorphism $\pi\sigma$ of the Lie algebra $\mathfrak b$; as $\pi\sigma=id$, this pull-back
is the adjoint module itself, therefore $\pi D\sigma\in\mathfrak{der}(\mathfrak b)$. Hence $D\mapsto \pi D\sigma$ defines a map 
$\mathfrak{der}(\mathfrak a)\to\mathfrak{der}(\mathfrak b)$, which obviously restricts to the parts of degree $>1$. 

Similarly, if $E\in\mathfrak{der}(\mathfrak b)[>1]$, then $\sigma E\pi$ is a derivation of $\mathfrak a$ with values in the pull-back 
of the submodule $\mathfrak a[>1]$ of the adjoint module $\mathfrak a$ by the endomorphism $\sigma\pi$ of $\mathfrak a$. 
This pull-back coincides with the submodule $\mathfrak a[>1]$ of the adjoint module $\mathfrak a$ itself, therefore 
$\sigma E\pi$ is a derivation of $\mathfrak a$ in the adjoint module $\mathfrak a$, hence $\sigma E\pi\in \mathfrak{der}(\mathfrak a)$. 

This implies that $D\mapsto \pi D\sigma$ and $E\mapsto \sigma E\pi$ are well-defined maps relating 
$\mathfrak{der}(\mathfrak a)[>1]$ and $\mathfrak{der}(\mathfrak b)[>1]$, which will be denoted $p$ and $s$. 

For $E\in \mathfrak{der}(\mathfrak a)[>1]$, $ps(E)=\pi\sigma E\pi\sigma=E$ as $\pi\sigma=id$. Therefore $ps=id$. 

For $D\in \mathfrak{der}(\mathfrak b)[>1]$, one has 
\begin{equation}\label{sp:D}
    sp(D)=\sigma\pi D\sigma\pi=D\sigma\pi
\end{equation} 
as $D$ is valued in $\mathfrak a[1]$
and the restriction of $\sigma\pi$ to this space is the identity. Moreover, 
$D(\mathrm{ker}\pi)\subset D(Z(\mathfrak a))\subset Z(\mathfrak a)$  where the first inclusion follows from
Lem. \ref{lem:general}(a), and the last inclusion follows from the fact that $D$ is a Lie algebra derivation; 
and since $D$ is supported in degree 1, $D(\mathrm{ker}\pi)\subset\mathfrak a[>1]$. Since $Z(\mathfrak a)$
is supported in degree 0, it follows that 
\begin{equation}\label{temp}
    D(\mathrm{ker}\pi)=0.
\end{equation} 
The equality $\mathrm{ker}\sigma\pi=\mathrm{ker}\pi$ follows from the injectivity of $\sigma$. Combining it with
\eqref{temp}, one derives 
\begin{equation}\label{temp2}
    D(\mathrm{ker}\sigma\pi)=0.
\end{equation} 
Since $\pi\sigma=id$, one has $(\sigma\pi)^2=\sigma\pi$, therefore 
$\sigma\pi \circ (id-\sigma\pi)=0$, therefore $\mathrm{im}(id-\sigma\pi)\subset \mathrm{ker}\sigma\pi$, which together with 
\eqref{temp2} implies $D\circ (id-\sigma\pi)=0$, therefore 
\begin{equation}\label{temp3}
D\in \mathfrak{der}(\mathfrak a)[>1],\quad D\sigma\pi=D.     
\end{equation}
Together with \eqref{sp:D}, this implies
$sp(D)=D$. Therefore $sp=id$. 

(b) Let $D\in\mathfrak{der}(\mathfrak a)[>0]$. Then $E=\pi D\sigma$. Then $E\pi=\pi D\sigma\pi=\pi D$ by \eqref{temp3}, hence 
$\pi$ intertwines $D$ and $E$. On the other hand, $D=\sigma E\pi$, hence $D\sigma=\sigma E\pi\sigma=\sigma E$ since $\pi\sigma=id$. 

(c) Let $a\in\mathfrak a[>1]$. Then $s(\mathrm{ad}_a)=\pi \mathrm{ad}_a\sigma=\mathrm{ad}_{\pi(a)}\pi\sigma=\mathrm{ad}_{\pi(a)}$
as $\pi\sigma=id$. Let $b\in\mathfrak b[>1]$. Then $p(\mathrm{ad}_b)=\sigma\mathrm{ad}_b\pi
=\mathrm{ad}_{\sigma(b)}\sigma \pi=\mathrm{ad}_{\sigma(b)}$ as $\mathrm{ad}_{\sigma(b)}$ is in 
$\mathfrak{der}(\mathfrak b)[>1]$ and by \eqref{temp3}. 
\end{proof}

\begin{lemma}\label{lem:indpce}
Let $\pi : \mathfrak a\to\mathfrak b$ be a morphism as in Lem. \ref{lem:general}, satisfying the equivalent conditions of 
 Lem. \ref{lem:general}(b). Then the bijection of Lem. \ref{lem:corresp}(b)
%, as well as the map $\mathfrak b[>1]\to\mathfrak a$, $b\mapsto \sigma(b)$ used in (c), 
 does not depend on a choice of $\sigma$. 
\end{lemma}

\begin{proof}
Let $\sigma$ and $\sigma'$ be graded sections of $\tau$. Then $\pi\sigma=\pi\sigma'=id$, therefore
$\pi\circ(\sigma'-\sigma)=0$, therefore $\mathrm{im}(\sigma'-\sigma)\subset \mathrm{ker}\pi\subset Z(\mathfrak a)$. 
If $D\in \mathfrak{der}(\mathfrak a)$, then the restriction of $D$ to $Z(\mathfrak a)$ is 0. This implies 
that the restriction of $D$ to $\mathrm{im}(\sigma'-\sigma)$ is zero, therefore $D\circ (\sigma'-\sigma)=0$. 
Hence $D\sigma'=D\sigma$, therefore $\pi D\sigma=\pi D\sigma'$, therefore the maps 
$\mathfrak{der}(\mathfrak a)[>1]\to \mathfrak{der}(\mathfrak b)[>1]$ corresponding to $\sigma$ and $\sigma'$
are the same. 
 \end{proof}

\subsection{Characterization of $\mathfrak{grt}_1$ in terms of the Lie algebras $\mathfrak{t}_3$ and $\mathfrak{P}_5$}
\label{sect:variant}

\begin{corollary}\label{cor:main}
The subset $\mathfrak{grt}_1\subset\mathfrak G$ is equal to the subset of $\mathfrak G_\inert$ of all elements  
$C$ such that there exists $\overline D\in\mathfrak{der}(\mathfrak P_5)$ and $\overline X\in\mathfrak P_5$, 
satisfying 

(a) the morphism $\mathfrak t_3\to\mathfrak P_5$ given by $\pi_4\circ (x\mapsto x^{1,2,3})$ intertwines 
$D_C$ and $\overline D$, and 

(b) the morphism $\mathfrak t_3\to\mathfrak P_5$ given by $\pi_4\circ (x\mapsto x^{1,2,4})$ intertwines $D_C$ 
and $\overline D+\mathrm{ad}_{\overline X}$. 
\end{corollary}

\begin{proof}
The map $\pi_4:\mathfrak t_4\to\mathfrak P_5$ is a surjective morphism of positively graded Lie algebras, which is an 
isomorphism in degree $>1$. The center $Z(\mathfrak t_4)$ is supported in degree 1. Therefore it satisfies the equivalent 
conditions of Lem. \ref{lem:general}(b). Let $\tilde\sigma :\mathfrak P_5[1]\to \mathfrak t_4[1]$ be the section given by 
$X_{12}\mapsto t_{12}$, $X_{13}\mapsto t_{13}$, $X_{23}\mapsto t_{23}$, $X_{14}\mapsto t_{13}$, $X_{24}\mapsto t_{23}$
(and therefore $X_{15}\mapsto -t_{1,234}$, $X_{25}\mapsto -t_{2,134}$, $X_{34}\mapsto t_{34}-z_4$, 
$X_{35}\mapsto -t_{3,124}+z_4$, $X_{45}\mapsto -t_{4,123}+z_4$)
and let $\sigma : \mathfrak P_5\to \mathfrak t_4$ be the corresponding graded section of $\pi_4$, 
constructed as in Lem. \ref{lem9}, which is a Lie algebra morphism by Lem. \ref{lem9}. It then follows from Lem. \ref{lem:corresp} that
the assignments $p : D\mapsto \sigma D\pi_4$ and $s : \overline D\mapsto \pi_4D\sigma$ are inverse bijections relating
$\mathfrak{der}(\mathfrak t_4)[1]$ and $\mathfrak{der}(\mathfrak P_5)[>1]$, and that $p(\mathrm{ad}_X)=\mathrm{ad}_{\pi_4(X)}$ for 
$X\in\mathfrak t_4$ and $s(\mathrm{ad}_{\overline X})=\mathrm{ad}_{\sigma(\overline X)}$ for any $\overline X\in\mathfrak P_5$. 

Let $C\in \mathfrak{grt}_1$. Let  $X\in \mathfrak t_4$ and $D\in\mathfrak{der}(\mathfrak t_4)$ be the elements associated with 
$C$ by Thm. \ref{thm:main}. They are supported in degree $>1$ since so is $\mathfrak{grt}_1$. Let 
$\overline X:=\pi_4(X)\in\mathfrak P_5[>1]$ and $\overline D:=p(D)\in \mathfrak{der}(\mathfrak P_5)[>1]$. 

It follows from Thm. \ref{thm:main} that $\mathfrak t_3\to\mathfrak t_4$, $x\mapsto x^{1,2,3}$ intertwines
$D_C$ and $D$, and from Lem. \ref{lem:corresp}(b) that $\pi_4 : \mathfrak t_4\to\mathfrak P_5$ intertwines $D$ and $p(D)=\overline D$. 
Therefore $\pi_4\circ (x\mapsto x^{1,2,3}) : \mathfrak t_3\to\mathfrak P_5$ intertwines $D_C$ and $\overline D$. 
It follows that $(C,\overline D,\overline X)$ satisfy the condition (a) from the statement.  

It follows from Thm. \ref{thm:main} that $\mathfrak t_3\to\mathfrak t_4$, $x\mapsto x^{1,2,4}$ intertwines
$D_C$ and $D+\mathrm{ad}_X$, and from Lem. \ref{lem:corresp}(b) that $\pi_4 : \mathfrak t_4\to\mathfrak P_5$ intertwines 
$D+\mathrm{ad}_X$ and $p(D+\mathrm{ad}_X)=p(D)+p(\mathrm{ad}_X)=\overline D+\mathrm{ad}_{\overline X}$, where the second
equality follows from the definition of $\overline D$ and from the first part of Lem. \ref{lem:corresp}(c). 
Therefore $\pi_4\circ (x\mapsto x^{1,2,4}) : \mathfrak t_3\to\mathfrak P_5$ intertwines $D_C$ and 
$\overline D+\mathrm{ad}_{\overline X}$. It follows that $(C,\overline D,\overline X)$ satisfy the condition (b) from the statement.  
All this implies that 
\begin{align}\label{direct:incl}
&\text{if $C\in\mathfrak{grt}_1$, then there exist $\overline D\in\mathfrak{der}(\mathfrak P_5)$ and $\overline X\in\mathfrak P_5$, 
satisfying }    
\\ & \nonumber\text{the conditions (a) and (b) from the statement.}\end{align}

Let now $C\in\mathfrak G_\inert$ and let $\overline D\in\mathfrak{der}(\mathfrak P_5)$ and $\overline X\in\mathfrak P_5$, 
satisfying the conditions (a) and (b) from the statement, i.e. $\pi_4\circ (x\mapsto x^{1,2,3})$ intertwines 
$D_C$ and $\overline D$ and $\pi_4\circ (x\mapsto x^{1,2,4})$ intertwines $D_C$ and $\overline D+\mathrm{ad}_{\overline X}$.  
One may assume the degree supports of 
$\overline D$ and $\overline X$ to be contained in the degree support of $C$, therefore in $\{d|d>1\}$. 
Let then $D:=s(\overline D)$ and $X:=\sigma(\overline X)$. Then $D\in\mathfrak{der}(\mathfrak t_4)$, $X\in\mathfrak t_4$, 
and $\sigma$ intertwines $\overline D$ and $s(\overline D)=D$ (by the second part of Lem. \ref{lem:corresp}(b) and the definition of 
$D$) as well as $\mathrm{ad}_{\overline X}$ and $s(\mathrm{ad}_{\overline X})=\mathrm{ad}_{\sigma(\overline X)}=\mathrm{ad}_X$
(by the second part of Lem. \ref{lem:corresp}(b), Lem. \ref{lem:corresp}(c) and the definition of $X$). 
All this implies that $\sigma\pi_4\circ (x\mapsto x^{1,2,3})$ intertwines $D_C$ and $D$ and that 
$\sigma\pi_4\circ (x\mapsto x^{1,2,4})$ intertwines $D_C$ and $D+\mathrm{ad}_X$. Since $\sigma$ takes 
$X_{ij}$ to $t_{ij}$ for $ij$ any of the pairs $12,13,23,14,24,34$, one has 
$\sigma\pi_4\circ (x\mapsto x^{1,2,3})=(x\mapsto x^{1,2,3})$ and 
$\sigma\pi_4\circ (x\mapsto x^{1,2,4})=(x\mapsto x^{1,2,4})$ (equalities of morphisms 
$\mathfrak t_3\to\mathfrak t_4$). It follows that $(x\mapsto x^{1,2,3})$ intertwines $D_C$ and $D$ and that 
$(x\mapsto x^{1,2,4})$ intertwines $D_C$ and $D+\mathrm{ad}_X$, therefore that $C$ satisfies conditions (a) and (b) from 
Thm. \ref{thm:main}, which by this theorem implies $C\in\mathfrak{grt}_1$.
Therefore 
\begin{align}\label{inverse:incl}
&\text{if $C\in\mathfrak G_\inert$ is such that for some $\overline D\in\mathfrak{der}(\mathfrak P_5)$ and $\overline X\in\mathfrak P_5$}    
\\ & \nonumber\text{the conditions (a) and (b) from the statement are satisfied, then $C\in\mathfrak{grt}_1$. }\end{align}
The statement then follows from the combination of \eqref{direct:incl} and \eqref{inverse:incl}. 
\end{proof}

\section{Characterization via Lie algebras of infinitesimal framed braids on the sphere}\label{sec:2}

This section is devoted to the proof of Thm. \ref{thm:main:2}. §\ref{sec:2:1} is devoted to the family of Lie algebras 
$\mathfrak P_{\vec n}$, $n \geq3$ and to the morphisms relating them. In §\ref{sect:2:2}, we discuss the relation 
with the Lie algebras $\mathfrak t_n$, $n \geq2$  from \S\ref{sect:1}. After proving an auxiliary result in §\ref{sect:2:3}, 
we formulate and prove Thm. \ref{thm:main:2} on a characterization of the Lie algebra $\mathfrak{grt}_1$ in terms of the 
Lie algebras $\mathfrak P_{\vec n}$ (§\ref{sect:2:4}). 

\subsection{Lie algebras of infinitesimal framed braids on the sphere}\label{sec:2:1}

Let $n\geq2$ and let $\mathfrak P_{\vec{n+1}}$ be the Lie algebra with generators $\vec X_{ij}$ where $i,j\in[1,n+1]$
and relations (a) $\vec X_{ji}=\vec X_{ij}$ for any $i,j$, (b) $[\vec X_{ij},\vec X_{kl}]=0$ for any distinct
$i,j,k,l$, (c) $\sum_j\vec X_{ij}=0$ for any $i$, (d) $\vec X_{ii}$ is central. This is a positively graded Lie algebra, 
each $\vec X_{ij}$ being attributed degree 1. 

\begin{lemma}\label{lem:descr:center:pn}
(a) There is a unique Lie algebra morphism $\vec\pi_n : \mathfrak P_{\vec{n+1}}\to\mathfrak P_{n+1}$ induced by 
$\vec X_{ij}\to X_{ij}$ for any $i,j$. 

(b) The kernel of $\vec\pi_n$  is supported in degree $1$, is equal to the center 
$Z(\mathfrak P_{\vec{n+1}})$, and is freely generated by $\vec X_{11},\ldots,\vec X_{n+1,n+1}$ as a $\mathbf k$-module.   
\end{lemma}

\begin{proof}
(a) follows from the presentations of $\mathfrak P_{\vec{n+1}}$ and $\mathfrak P_{{n+1}}$. 
(b) These presentations also imply that $\mathrm{ker}\vec\pi_n$
is the ideal generated by $\vec X_{11},\ldots,\vec X_{n+1,n+1}$, which since these elements are central
is also the $\mathbf k$-submodule generated by them, and is therefore supported in degree 1. 
A study of the map $\mathfrak P_{\vec{n+1}}\to\mathfrak P_{{n+1}}$ then implies that 
$\mathrm{ker}(\mathfrak P_{\vec{n+1}}\to\mathfrak P_{n+1})$ is freely generated by 
$\vec X_{11},\ldots,\vec X_{n+1,n+1}$ as a $\mathbf k$-module.   

It follows from Lem. \ref{lem:general} applied to $\pi_n : \mathfrak t_n\to \mathfrak P_{n+1}$ 
together with the fact that $Z(\mathfrak t_n)$ is supported in degree 1 that $Z(\mathfrak P_{n+1})
=\pi_n(Z(\mathfrak t_n))$, therefore that $Z(\mathfrak P_{n+1})=0$. Using this equality and applying Lem. \ref{lem:general} to 
$\vec\pi_n : \mathfrak P_{\vec{n+1}}\to \mathfrak P_{n+1}$, one then obtains 
$Z(\mathfrak P_{\vec{n+1}})=\mathrm{ker}\pi_{\vec n}$. 
\end{proof}

Let $n,m\geq2$. A map $\phi : [1,m+1]\to [1,n+1]$ may as well be identified with the sequence 
$\phi^{-1}(1),\ldots,\phi^{-1}(n+1)$ of its fibers. Such a map 
defines a Lie algebra morphism $\mathfrak P_{\overrightarrow{n+1}}\to\mathfrak P_{\vec{m+1}}$, $x\mapsto x^\phi$, by 
$\vec X_{ij}\mapsto \sum_{i',j':\phi(i')=i,\phi(j')=j}\vec X_{i'j'}$, homogeneous of degree 0.

\subsection{Relation with Lie algebras of infinitesimal braids on the plane}\label{sect:2:2}

\begin{lemma}\label{lem:friday}
(a) For $n\geq 2$ there is a unique Lie algebra morphism $\underline\pi_n : \mathfrak P_{\vec{n+1}}\to \mathfrak t_n$, such that 
$$
\underline\pi_n : \vec X_{ii}\mapsto 0 \text{ for } i\in[1,n],\vec X_{n+1,n+1}\mapsto 2z_n, \vec X_{ij}\mapsto t_{ij}  \text{ for } i\neq j\in[1,n],
\vec X_{i,n+1}\mapsto -t_{i,1...\check i...n}  \text{ for } i\in[1,n].  
$$

(b) The following diagrams are commutative 
$$
\xymatrix@C=4em
{ \mathfrak P_{\vec4} \ar^{x\mapsto x^{1,2,3,45}}[r]\ar_{\underline\pi_3}[d]&\mathfrak P_{\vec5}\ar^{\underline\pi_4}[d] 
\\\mathfrak t_3\ar^{ x\mapsto x^{1,2,3}}[r] & \mathfrak t_4}
\text{ and }
\xymatrix@C=4em
{ \mathfrak P_{\vec4} \ar^{x\mapsto x^{1,2,4,35}}[r]\ar_{\underline\pi_3}[d]&\mathfrak P_{\vec5}\ar^{\underline\pi_4}[d] 
\\\mathfrak t_3\ar^{ x\mapsto x^{1,2,4}}[r] & \mathfrak t_4}
$$ 
\end{lemma}

\begin{proof}
(a) The proposed assignment  $\vec X_{ij}\mapsto \underline\pi_n(\vec X_{ij})$ induces a linear map 
$\mathfrak P_{\vec{n+1}}[1]\to \mathfrak t_n[1]$ in degree 1: the relation $\sum_j\vec X_{ij}=0$ is obviously preserved if
$i\neq n+1$, and it is preserved as a consequence of $2z_n=\sum_{i=1}^n t_{i,1...\check i...n}$ if $i=n+1$. 
Moreover, the composition of $\vec X_{ij}\mapsto \underline\pi_n(\vec X_{ij})$ with the projection 
$\mathfrak t_n\to\mathfrak P_n$ corresponds to the canonical Lie algebra morphism $\mathfrak P_{\vec{n+1}}\to \mathfrak P_{n+1}$. 
The result follows.

(b) For $i\neq j\in[1,3]$, one has $\underline\pi_4(X_{ij}^{1,2,3,45})=\underline\pi_4(X_{ij})=t_{ij}=
t_{ij}^{1,2,3}=\underline\pi_3(X_{ij})^{1,2,3}$ and for $i\in[1,3]$, one has 
$\underline\pi_4(X_{ii}^{1,2,3,45})=\underline\pi_4(0)=0=0^{1,2,3}=\underline\pi_3(X_{ii})^{1,2,3}$, which
implies the commutativity of the left diagram. 

Similarly, for $i\neq j\in[1,3]$, one has $\underline\pi_4(X_{ij}^{1,2,4,35})=\underline\pi_4(X_{\beta^{-1}(ij)})=t_{\beta^{-1}(ij)}=
t_{ij}^{1,2,4}=\underline\pi_3(X_{ij})^{1,2,4}$, where $\beta$ is as in \eqref{def:alpha:beta}, and for $i\in[1,3]$, one has 
$\underline\pi_4(X_{ii}^{1,2,4,35})=\underline\pi_4(0)=0=0^{1,2,4}=\underline\pi_3(X_{ii})^{1,2,4}$, which
implies the commutativity of the right diagram. 
\end{proof}

\begin{lemma}\label{underlinepi:lem8b:part}
For $n\geq2$, the morphism $\underline\pi_n : \mathfrak P_{\vec{n+1}}\to\mathfrak t_n$ is graded, surjective, 
and is an isomorphism in degree $>1$. 
%satisfies the equivalent conditions of Lem. \ref{lem:general}(b). 
\end{lemma}

\begin{proof} Since the image of $\underline\pi_n$ contains a generating set of $\mathfrak t_n$, it is surjective.  
Its composition with $\pi_n : \mathfrak t_n\to \mathfrak P_{n+1}$ is the
projection $\mathfrak P_{\vec{n+1}}\to\mathfrak P_{n+1}$, which was shown in Lem. \ref{lem:descr:center:pn}
to be an isomorphism in degrees $>1$; since this is also the case of $\pi_n$, it follows that 
$\underline\pi_n$ is an isomorphism in degrees $>1$. 
%By Lem. \ref{lem:descr:center:pn}, the center of $\mathfrak P_{\vec{n+1}}$ is supported in degree 1. 
\end{proof}

Let $\alpha : \{1,2,4\}\to [1,3]$ %and $\beta : \{12,14,24\}\to\{12,13,23\}$ 
be the bijection defined by 
\begin{equation}\label{def:alpha:beta}
    \alpha : 1,2,4\mapsto 1,2,3.%\quad \beta : 12,14,24\mapsto 12,13,23. 
\end{equation}

\begin{lemma}\label{lem:section:A}
Let $A : \{12,13,23\}\times[1,3]\to \mathbf k$, $(ij,k)\mapsto A_{ij}^k$ be a map.

(a) There is are uniquely defined graded Lie algebra morphisms $\underline\sigma^A : \mathfrak t_3\to\mathfrak P_{\vec4}$
and $\underline\tau^A : \mathfrak t_4\to\mathfrak P_{\vec5}$ and
$\underline\theta^A : \mathfrak t_4\to\mathfrak P_{\vec5}$, such that 
$$
\textstyle\underline\sigma^A : t_{ij}\mapsto \vec X_{ij}+\sum_{k=1}^3 A_{ij}^k \vec X_{kk}\text{ for }ij\in\{12,13,23\}, 
$$
$$
\textstyle\underline\tau^A : t_{ij}\mapsto \vec X_{ij}+\sum_{k=1}^3 A_{ij}^k \vec X_{kk}\text{ for }ij\in\{12,13,23\},\quad 
t_{i4}\mapsto \vec X_{i4}\text{ for }i\in\{1,2,3\}; 
$$
and
$$
\textstyle\underline\theta^A : t_{ij}\mapsto \vec X_{ij}+\sum_{k\in\{1,2,4\}} B_{ij}^k \vec X_{kk}\text{ for }ij\in\{12,14,24\},\quad 
t_{i3}\mapsto \vec X_{i3}\text{ for }i\in\{1,2,4\}, 
$$
where $B : 
\{12,14,24\}\times\{1,2,4\}\to \mathbf k$ is the map defined by $B_{ij}^k:=A_{\alpha(i)\alpha(j))}^{\alpha(k)}$, where 
$\alpha$ is as in \eqref{def:alpha:beta}. 

(b) $\underline\sigma^A$ is a section of $\underline\pi_3$ and $\underline\tau^A,\underline\theta^A$ are both sections of 
$\underline\pi_4$. 

(c) The following diagrams are commutative 
$$
\xymatrix@C=4em
{ \mathfrak t_3 \ar^{x\mapsto x^{1,2,3}}[r]\ar_{\underline\sigma^A}[d]&\mathfrak t_4\ar^{\underline\tau^A}[d] 
\\\mathfrak P_{\vec4}\ar^{x\mapsto x^{1,2,3,45}}[r] & \mathfrak P_{\vec5}} \text{ and }
\xymatrix@C=4em
{ \mathfrak t_3 \ar^{x\mapsto x^{1,2,4}}[r]\ar_{\underline\sigma^A}[d]&\mathfrak t_4\ar^{\underline\theta^A}[d] 
\\\mathfrak P_{\vec4}\ar^{x\mapsto x^{1,2,4,35}}[r] & \mathfrak P_{\vec5}}  
$$
\end{lemma}

\begin{proof}
(a) Let $n\in\{2,3\}$. The proposed assignments define linear maps $\mathfrak t_n[1]\to \mathfrak P_{\vec{n+1}}[1]$
as $(t_{ij})_{i,j:1\leq i<j\leq n}$ is a basis of $\mathfrak t_n[1]$. Their compositions with the projection 
$\mathfrak P_{\vec{n+1}}\to\mathfrak P_{{n+1}}$ corresponds to the morphism $\mathfrak t_n\to \mathfrak P_{{n+1}}$, 
which implies the result. 

(b) One has $\underline\pi_3\underline\sigma^A(t_{ij})=\underline\pi_3(\vec X_{ij}+\sum_{k=1}^3 A_{ij}^k \vec X_{kk})
=t_{ij}$ for $ij\in\{12,13,23\}$, which implies $\underline\pi_3\underline\sigma^A=id$. One also has 
$\underline\pi_4\underline\tau^A(t_{ij})=\underline\pi_4(\vec X_{ij}+\sum_{k=1}^3 A_{ij}^k \vec X_{kk})=t_{ij}$
for $ij\in \{12,13,23\}$ and $\underline\pi_4\underline\tau^A(t_{i4})=\underline\pi_4(\vec X_{i4})=t_{i4}$
for $i\in[1,3]$, hence $\underline\pi_4\underline\tau^A=id$. Finally, one has 
$\underline\pi_4\underline\theta^A(t_{ij})=\underline\pi_4(\vec X_{ij}+\sum_{k\in\{1,2,4\}} B_{ij}^k \vec X_{kk})=t_{ij}$
for $ij\in \{12,14,24\}$ 
and $\underline\pi_4\underline\tau^A(t_{i3})=\underline\pi_4(\vec X_{i3})=t_{i3}$
for $i\in\{1,2,4\}$, hence $\underline\pi_4\underline\theta^A=id$.

(c) For $i<j\in[1,3]$, one has $\underline\sigma_A(t_{ij})^{1,2,3,45}=(\vec X_{ij}+\sum_k A_{ij}^k \vec X_{kk})^{1,2,3,45}
=\vec X_{ij}+\sum_k A_{ij}^k \vec X_{kk}=\underline\tau_A(t_{ij})=\underline\tau_A(t_{ij}^{1,2,3})$, which implies the commutativity 
of the first diagram. The commutativity of the second diagram follows from the equalities  $x^{1,2,4}=(x^{1,2,3})^{1,2,4,3}$
for $x\in\mathfrak t_3$ (equality in $\mathfrak t_4$), $y^{1,2,4,35}=(y^{1,2,3,45})^{1,2,4,3,5}$ for $y\in\mathfrak P_{\vec4}$
(equality in $\mathfrak P_{\vec5}$), and from composing the first diagram with the commutative diagram 
$$
\xymatrix@C=4em
{\mathfrak t_4\ar^{x\mapsto x^{1,2,4,3}}[r]\ar_{\underline\tau_A}[d]&\mathfrak t_4\ar^{\underline\theta_A}[d] \\
\mathfrak P_{\vec5}\ar^{x\mapsto x^{1,2,4,3,5}}[r]& \mathfrak P_{\vec5}}
$$
\end{proof}

\begin{remark}
For $\Sigma$ a surface, let $\mathrm{Cf}_n(\Sigma)$ be the configuration space of $n$ distinct points on $\Sigma$ and 
let $\mathrm{Cf}_{\vec n}(\Sigma)$ be the space of $n$ pairs $(x_1,\vec v_1),...,(x_n,\vec v_n)$, where 
$(x_1,...,x_n)\in \mathrm{Cf}_n(\Sigma)$ and $\vec v_i\in T_{x_i}\Sigma\smallsetminus\{0\}$ is a nonzero tangent vector
at $x_i$ for any $i$. Then $P_n(\Sigma):=\pi_1\mathrm{Cf}_n(\Sigma)$ (resp. $P_{\vec n}(\Sigma):=\pi_1\mathrm{Cf}_{\vec n}(\Sigma)$) 
is the pure (resp. pure framed) braid group of $\Sigma$ and there is a central extension 
$0\to\mathbb Z^n\to P_{\vec n}(\Sigma)\to P_n(\Sigma)\to 1$. Denoting by $\mathrm{gr}\Gamma$ the graded Lie algebra 
associated with the lower central series filtration of a group $\Gamma$, then one can prove that 
$\mathfrak t_n=\mathrm{gr}P_n(\mathbb R^2)$ while $\mathfrak P_{\vec{n+1}}=\mathrm{gr}P_{\vec{n+1}}(S^2)$. 
Then $\underline\pi_n:\mathfrak P_{\vec{n+1}}\to \mathfrak t_n$ corresponds to the group morphism 
$P_{\vec{n+1}}(S^2)\to P_n(\mathbb R^2)$ induced by the (well-defined) map 
$\mathrm{Cf}_{\vec{n+1}}(S^2)\to \mathrm{Cf}_n(\mathbb R^2)/\mathbb R^2$
given by $((x_1,\vec v_1),...,(x_{n+1},\vec v_{n+1}))\mapsto [\theta(x_1),...,\theta(x_n)]$, 
where $\theta$ is any homographic transformation of $S^2\simeq\mathbb P^1_{\mathbb C}$ taking $(x_{n+1},\vec v_{n+1})$
to $(\infty,\vec v_\infty)$ and $[-]$ is the class modulo the additive action of $\mathbb R^2$, 
$\vec v_\infty\in T_{\infty}\Sigma\smallsetminus\{0\}$ being a fixed nonzero vector. 
\end{remark}

\subsection{An auxiliary result}\label{sect:2:3}

\begin{lemma}\label{lem:LAs}
Let 
$$
\xymatrix{\mathfrak a\ar^i[r]\ar_u[d]&\mathfrak b\ar^v[d]\\\mathfrak c\ar_j[r]&\mathfrak d}
$$
be a commutative diagram of Lie algebra and $\alpha$ (resp. $\beta,\gamma,\delta$) be derivations of 
$\mathfrak a$ (resp. $\mathfrak b,\mathfrak c,\mathfrak d$) such that $i$ (resp. $u,v$) intertwines 
$\alpha$ and $\beta$ (resp. $\alpha$ and $\gamma$, $\beta$ and $\delta$).

Then $u(\mathfrak a)\subset\mathfrak c$ is a Lie subalgebra which is stable under $\gamma$, and the map
$j\vert_{u(\mathfrak a)} : u(\mathfrak a)\to \mathfrak d$ intertwines the derivations 
$\gamma\vert_{u(\mathfrak a)}$ and $\delta$. 
\end{lemma}

\begin{proof}
The subspace $u(\mathfrak a)$ is a Lie subalgebra of $\mathfrak c$ as it is the image of a Lie algebra morphism. 
Let $c\in u(\mathfrak a)$ and let $a\in\mathfrak a$ be a lift of $c$ by $u$. Then: 
$$
\gamma(c)=\gamma(u(a))=u(\alpha(a))\in u(\mathfrak a),
$$
where the first equality follows from the definition of $a$ and the second equality follows from the intertwining property of $u$; 
and 
$$
j(\gamma(c))=j(\gamma(u(a)))=j(u (\alpha(a)))=v(i(\alpha(a)))=v(\beta(i(a)))=\delta(v(i(a)))=\delta(j(u(a)))=\delta(j(c)), 
%=\delta(v(i(a)))=\delta(j(u(a)))=\delta(j(c)), 
$$
where the first and last equalities follow from the definition of $a$, the second (resp. fourth, fifth) 
equality follows from the intertwining property of $u$ (resp. $i$, $v$), and the third and sixth equalities follow 
from the commutativity of the diagram. 

The first (resp. second) identity implies the first (resp. second) part of the statement. 
\end{proof}

\subsection{Characterization of $\mathfrak{grt}_1$ in terms of the Lie algebras $\mathfrak P_{\vec4}$ and $\mathfrak P_{\vec5}$}
\label{sect:2:4}

\begin{lemma}\label{def:D:vecC}
For $C\in \mathfrak G_{\inert}$ (see \S\ref{sect:12}), there is a unique derivation $\vec D_C$ of $\mathfrak P_{\vec 4}$, such that 
$$
\vec X_{11},\ldots,\vec X_{44}\mapsto 0,\quad \vec X_{12},\vec X_{34}\mapsto 0,\quad \vec X_{23},\vec X_{14}\mapsto [C(\vec X_{32},\vec X_{21}),\vec X_{23}],\quad 
\vec X_{13},\vec X_{24}\mapsto [h_C(\vec X_{32},\vec X_{21}),\vec X_{13}]. 
$$
\end{lemma}

\begin{proof} Let us check that the defining relations of $\mathfrak P_{\vec 4}$ are preserved under $\vec D_C$. 
Since $\vec X_{ii}\mapsto 0$ for any $i$, the relation of centrality of $\vec X_{ii}$ is preserved. 
Since the images of $\vec X_{ij}$ and $\vec X_{kl}$ are equal 
for distinct $i,j,k,l$, the relation 
$[\vec X_{ij},\vec X_{kl}]=0$ is preserved. Finally the specialization $e_0\mapsto\vec X_{23},
e_1\mapsto\vec X_{12}$ of the relation $[C(e_0,e_1),e_0]+[h_C(e_0,e_1),e_\infty]=0$ 
implies that the linear relations $\sum_j\vec X_{ij}=0$ are preserved for any $i$, as this specialization 
takes  $e_\infty$ to $\vec X_{13}+\sum_{a=1}^4\mathbf k\vec X_{aa}$.  
\end{proof}

\begin{lemma}\label{lem:inter:DC}
Let $\sigma:\mathfrak t_3\to\mathfrak P_{\vec4}$ be a graded section of $\underline\pi_3:\mathfrak P_{\vec4}\to \mathfrak t_3$ and 
let $C\in\mathfrak G_\inert$. 

(a) The bijection $p:=\mathfrak{der}(\mathfrak P_{\vec4})[>1]\to \mathfrak{der}(\mathfrak t_{3})[>1]$ from Lem. \ref{lem:corresp}(a) 
corresponding to  $\sigma$ takes $\vec D_C$ to $D_C$, i.e. $D_C=p(\vec{D}_C)$. 

(b) The morphism $\underline\pi_3 : \mathfrak P_{\vec4}\to\mathfrak t_3$ intertwines $\vec D_C$ and $D_C$.

(c) $\sigma$ intertwines $D_C$ and $\vec D_C$. 
\end{lemma}

\begin{proof} (a) There exists scalars $(\sigma_{ij}^k)_{ij,k}$ such that 
$\sigma(t_{ij})=\vec X_{ij}+\sum_{k=1}^4 \sigma_{ij}^k \vec X_{kk}$. Then 
$\underline\pi_3 \vec D_C\sigma(t_{12})=\underline\pi_3 \vec D_C(\vec X_{12}+\sum_{i=1}^4 \sigma_{12}^i \vec X_{ii})=0=D_C(t_{12})$. 
Similarly, $\underline\pi_3 \vec D_C\sigma(t_{23})=\underline\pi_3 \vec D_C(\vec X_{23}+\sum_{i=1}^4 \sigma_{23}^i \vec X_{ii})=\underline\pi_3([C(\vec X_{32},\vec X_{21}),\vec X_{23}])=[C(t_{32},t_{21}),t_{23}]=D_C(t_{12})$ and 
$\underline\pi_3 \vec D_C\sigma(t_{13})=\underline\pi_3 \vec D_C(\vec X_{13}+\sum_{i=1}^4 \sigma_{13}^i 
\vec X_{ii})=\underline\pi_3([h_C(\vec X_{32},\vec X_{21}),\vec X_{13}])=[h_C(t_{32},t_{21}),t_{13}]=D_C(t_{13})$. 
Therefore $D_C=\underline\pi_3 \vec D_C\sigma=p(\vec D_C)$. 
(b) and (c) follow from (a) and from Lem. \ref{lem:corresp}(b). 
\end{proof}

We will show: 
\begin{theorem}\label{thm:main:2}
The subset $\mathfrak{grt}_1$ of $\mathfrak G$ 
is equal to the set of all elements $C\in\mathfrak G_\inert$, such that there exists a derivation 
$\vec D$ of $\mathfrak P_{\vec 5}$ and $\vec X\in\mathfrak P_{\vec 5}$, such that: 

\emph{\((a')\)} 
the morphism $\mathfrak P_{\vec 4}\to\mathfrak P_{\vec 5}$, $x\mapsto x^{1,2,3,45}$ intertwines $\vec D_C$ and $\vec D$, and 

\emph{\((b')\)} 
the morphism $\mathfrak P_{\vec 4}\to\mathfrak P_{\vec 5}$, $x\mapsto x^{1,2,4,35}$ intertwines $\vec D_C$ and 
$\vec D+\mathrm{ad}_{\vec X}$. 
\end{theorem}

\begin{proof}
Let $\mathfrak Y$ be the subset of $\mathfrak G_\inert$ defined by the said conditions and recall the subset 
$\mathfrak X$ from the proof of Thm. \ref{thm:main}. We will prove
\begin{equation}\label{incl:Y:X}
    \mathfrak Y\subset \mathfrak X. 
\end{equation}
Let $C\in\mathfrak Y$. Let then $\vec D\in\mathfrak{der}(\mathfrak P_{\vec5})$ and $\vec X\in\mathfrak P_{\vec5}$ be such that 
$(C,\vec D,\vec X)$ satisfies conditions (a') and (b'). Since $C$ is supported in degree $>1$, one may assume the same on
$\vec D$ and $\vec X$. Recall that $\underline\pi_4 : \mathfrak P_{\vec5}\to \mathfrak t_4$ is surjective, an
isomorphism is degree $>1$ (Lem. \ref{underlinepi:lem8b:part}) 
and that $Z(\mathfrak P_{\vec5})$ is supported in degree 1 (see Lem. \ref{lem:descr:center:pn}). 
Let then $X:=\underline\pi_4(\vec X)\in\mathfrak t_4$ and $D:=\underline\pi_4 D\sigma$, where  
$\sigma$ is any graded section of $\underline\pi_3$. It follows from Lem. \ref{lem:corresp}(a) that $D$ is a derivation of 
$\mathfrak t_4$. We now show that $(C,D,X)$ satisfy conditions (a) and (b) from Thm. \ref{thm:main}. 

Consider the left commutative diagram of Lem. \ref{lem:friday}(b). The Lie algebras in this diagram are equipped with the 
derivations $\vec D_C$ and $\vec D$ for the top row, and $D_C$ and $D$ for the bottom row. 
It follows from Lem. \ref{lem:inter:DC}(b) that $\underline\pi_3$ intertwines $\vec D_C$ and $D_C$; it follows from 
the fact that $(C,\vec D,\vec X)$ satisfies conditions (a') 
that $x\mapsto x^{1,2,3,45}$ intertwines $\vec D_C$ and $\vec D$; and it follows from Lem. \ref{lem:corresp}(b) that 
$\underline\pi_4$ intertwines $\vec D$ and $D$. Lem. \ref{lem:LAs}, together with the surjectivity of 
$\underline\pi_3$ (see Lem. \ref{underlinepi:lem8b:part}), 
implies that $x\mapsto x^{1,2,3}$ intertwines $D_C$ and $D$, therefore that $(C,D,X)$ satisfies condition (a)
from Thm. \ref{thm:main}. 

Consider now the right commutative diagram of Lem. \ref{lem:friday}(b). Its Lie algebras are equipped with the 
derivations $\vec D_C$ and $\vec D+\mathrm{ad}_{\vec X}$ for the top row, and $D_C$ and $D+\mathrm{ad}_{X}$ for the bottom row. 
Recall again that that $\underline\pi_3$ intertwines $\vec D_C$ and $D_C$ (Lem. \ref{lem:inter:DC}(b)); it follows from 
the fact that $(C,\vec D,\vec X)$ satisfies conditions (a') 
that $x\mapsto x^{1,2,4,35}$ intertwines $\vec D_C$ and $\vec D+\mathrm{ad}_{\vec X}$; and it follows 
from Lem. \ref{lem:corresp}(b) together with Lem. \ref{lem:corresp}(b) and (c) that 
$\underline\pi_4$ intertwines $\vec D+\mathrm{ad}_{\vec X}$ and $D+\mathrm{ad}_{X}$. Lem. \ref{lem:LAs}, together with the 
surjectivity of $\underline\pi_3$ (see Lem. \ref{underlinepi:lem8b:part}), 
implies that $x\mapsto x^{1,2,4}$ intertwines $D_C$ and $D+\mathrm{ad}_{X}$, 
therefore that $(C,D,X)$ satisfies condition (b) from Thm. \ref{thm:main}. 

This shows \eqref{incl:Y:X}. 

Let us now show
\begin{equation}\label{incl:X:Y}
    \mathfrak X\subset \mathfrak Y. 
\end{equation}
Let $C\in\mathfrak Y$. Let $D\in\mathfrak{der}(\mathfrak t_{4})$ and $X\in\mathfrak t_{4}$ be such that 
$(C,D,X)$ satisfies conditions (a) and (b) from Thm. \ref{thm:hf}. Since $C$ is supported in degree $>1$, 
one may assume the same on $D$ and $X$.

Recall from Lem. \ref{lem:section:A} the section $\underline\tau^0$ of $\underline\pi_3$ corresponding to the zero 
map $0:\{12,13,23\}\times[1,3]\to \mathbf k$. Let $\vec X:=\underline\tau^0(X)$ and 
$\vec D:=\underline\tau^0 D\underline\pi_3$. Then $\vec X\in \mathfrak P_{\vec5}$
and by Lem. \ref{lem:corresp}(a), $\vec D\in \mathfrak{der}(\mathfrak P_{\vec5})$. Let us show that 
$(C,\vec D,\vec X)$ satisfies conditions 
(a') and (b').

Let $A : \{12,13,23\}\times[1,3]\to \mathbf k$ be a map and consider the left commutative diagram of Lie algebras in 
Lem. \ref{lem:section:A}(c). Its Lie algebras are equipped, for the top row, with the derivations $D_C$ and $D$, and for the 
bottom row, with the derivations $\vec D_C$ and $\vec D$. Since $\underline\sigma^A$ is a section of $\underline\pi_3$, it follows from 
Lem. \ref{lem:inter:DC}(c) that it intertwines $\vec D_C$ and $D_C$. It follows from the fact that $(C,D,X)$ satisfies condition (a) from Thm. 
\ref{thm:main} that $x\mapsto x^{1,2,3}$ intertwines $D_C$ and $D$. It follows from 
Lem. \ref{lem:corresp}(b) that $\underline\tau^A$ intertwines $D$ and $\underline\tau^A D\underline\pi_3$, which 
by the equality $\vec D=\underline\tau^A D\underline\pi_3$ (see Lem. \ref{lem:indpce}) implies that 
$\underline\tau^A$ intertwines $D$ and $\vec D$. 
Lem. \ref{lem:LAs} then implies that 
the restriction of $x\mapsto x^{1,2,3,45}$ to $\underline\sigma^A(\mathfrak t_3)$ intertwines 
$\vec D_C$ and $\vec D$. Therefore, the restriction of the linear map  
$\vec D\circ (x\mapsto x^{1,2,3,45})-(x\mapsto x^{1,2,3,45})\circ \vec D_C$ to 
$\underline\sigma^A(\mathfrak t_3)\subset \mathfrak P_{\vec4}$ is zero. 
Since 
$\sum_{A \in \mathbf k^{\{12,13,23\}\times[1,3]}} \underline\sigma^A(\mathfrak t_3)=\mathfrak P_{\vec4}$
(one already has $\underline\sigma^{A_1}(\mathfrak t_3)+\underline\sigma^{A_2}(\mathfrak t_3)
+\underline\sigma^{A_3}(\mathfrak t_3)=\mathfrak P_{\vec4}$ where $A_1,A_2,A_3$ are defined by 
$(A_s)_{ij}^k=\delta_{sk}\delta_{ij,12}$ for $s\in[1,3]$), this linear map is zero, therefore $x\mapsto x^{1,2,3,45}$ intertwines 
$\vec D_C$ and $\vec D$. Therefore $(C,\vec D,\vec X)$ satisfies (a').

Let $A : \{12,13,23\}\times[1,3]\to \mathbf k$ be a map and consider the right commutative diagram of Lie algebras in 
Lem. \ref{lem:section:A}(c). Its Lie algebras are equipped, for the top row, with the derivations $D_C$ and $D+\mathrm{ad}_X$, 
and for the bottom row, with the derivations $\vec D_C$ and $\vec D+\mathrm{ad}_{\vec X}$. 
Recall from the previous paragraph that $\underline\sigma^A$ intertwines $\vec D_C$ and $D_C$. It follows from the fact that 
$(C,D,X)$ satisfies condition (b) from Thm. \ref{thm:main} that $x\mapsto x^{1,2,4}$ intertwines $D_C$ and $D+\mathrm{ad}_X$. 
It follows from 
Lem. \ref{lem:corresp}(b) that $\underline\theta^A$ intertwines $D+\mathrm{ad}_X$ and 
$\underline\theta^A (D+\mathrm{ad}_X)\underline\pi_3$, which by the equality $\vec D=\underline\theta^A D\underline\pi_3$ 
(see Lem. \ref{lem:indpce}) and $\underline\theta^A \mathrm{ad}_X\underline\pi_3=\mathrm{ad}_{\vec X}$ (by Lem. \ref{lem:corresp}(c)) 
implies that $\underline\theta^A$ intertwines $D+\mathrm{ad}_X$ and $\vec D+\mathrm{ad}_{\vec X}$. 
Lem. \ref{lem:LAs} then implies that 
the restriction of $x\mapsto x^{1,2,4,35}$ to $\underline\sigma^A(\mathfrak t_3)$ intertwines 
$\vec D_C$ and $\vec D+\mathrm{ad}_{\vec X}$. Therefore, the restriction of the linear map  
$(\vec D+\mathrm{ad}_{\vec X})\circ (x\mapsto x^{1,2,4,35})-(x\mapsto x^{1,2,4,35})\circ \vec D_C$ to 
$\underline\sigma^A(\mathfrak t_3)\subset \mathfrak P_{\vec4}$ is zero. 
Since 
$\sum_{A \in \mathbf k^{\{12,13,23\}\times[1,3]}} \underline\sigma^A(\mathfrak t_3)=\mathfrak P_{\vec4}$
(see previous paragraph), this linear map is zero, therefore $x\mapsto x^{1,2,4,35}$ intertwines 
$\vec D_C$ and $\vec D+\mathrm{ad}_{\vec X}$. Therefore $(C,\vec D,\vec X)$ satisfies (b').

This shows \eqref{incl:X:Y}. The result then follows from the conjunction of 
\eqref{incl:Y:X}, \eqref{incl:X:Y} and Thm. \ref{thm:main}. 

\end{proof}

\section{Stabilizers of outer morphisms of Lie algebras}\label{sec:3}

In this section, we attach a "stabilizer" Lie algebra to a morphism of Lie algebras (§\ref{sect:3:1}). 
After introducing the necessary group-scheme theoretic material (§\ref{sect:3:2}), we introduce 
a group scheme version of this construction, in the particular situation of morphisms of pronilpotent 
Lie algebras (§\ref{sec:3:3}).  

\subsection{Stabilizer Lie algebras of morphisms of Lie algebras} \label{sect:3:1}

If $\mathfrak a$ is a Lie algebra, then the set $\mathfrak{inn}(\mathfrak a)$ of its inner derivations is an 
ideal of its Lie algebra $\mathfrak{der}(\mathfrak a)$ of derivations, and the corresponding quotient is denoted 
$\mathfrak{out}(\mathfrak a)$. 

\begin{definition}
If $\mu : \mathfrak a\to\mathfrak b$ is a Lie algebra morphism, and if $\underline D\in\mathfrak{out}(\mathfrak a)$, 
 $\underline E\in\mathfrak{out}(\mathfrak b)$, then $\mu$ is said to intertwine $\underline D,\underline E$ iff
there exist representatives $D,E$ of $\underline D,\underline E$ which are intertwined by $\mu$, i.e. such that 
$E\mu=\mu D$. 
\end{definition}

\begin{lemma}\label{lem:stab:las}
If $\mu : \mathfrak a\to\mathfrak b$ is a Lie algebra morphism, then 
$$\{(\underline D,\underline E)\in \mathfrak{out}(\mathfrak a)
\times \mathfrak{out}(\mathfrak b)|\text{$\mu$ intertwines $\underline D$ and $\underline E$ }\}$$ is a Lie subalgebra of 
$\mathfrak{out}(\mathfrak a)\times \mathfrak{out}(\mathfrak b)$, denoted $\mathfrak{stab}_{ \mathfrak{out}(\mathfrak a)
\times \mathfrak{out}(\mathfrak b)}(\overline\mu)$. 
\end{lemma}

\begin{proof}
The subspace $\{(D,E)|\mu D=E\mu\}$ is a Lie subalgebra of $\mathfrak{der}(\mathfrak a)\times\mathfrak{der}(\mathfrak b)$. 
The statement follows from fact that
$\mathfrak{stab}_{ \mathfrak{out}(\mathfrak a)\times \mathfrak{out}(\mathfrak b)}(\overline\mu)$ is its image under the 
projection $\mathfrak{der}(\mathfrak a)\times\mathfrak{der}(\mathfrak b)\to\mathfrak{out}(\mathfrak a)
\times \mathfrak{out}(\mathfrak b)$, which is a Lie algebra morphism.
\end{proof}

For $\mathfrak g\to \mathfrak{out}(\mathfrak a)\times \mathfrak{out}(\mathfrak b)$ a Lie algebra morphism, we denote by 
$\mathfrak{stab}_{ \mathfrak g}(\overline\mu)$ the preimage of $\mathfrak{stab}_{ \mathfrak{out}(\mathfrak a)\times 
\mathfrak{out}(\mathfrak b)}(\overline\mu)$ by this morphism. Note that $\mathfrak{stab}_{ \mathfrak{der}(\mathfrak a)\times 
\mathfrak{der}(\mathfrak b)}(\overline\mu)$ is the sum of $\{(D,E)|E\mu=\mu D\}$ with 
$\mathfrak{inn}(\mathfrak{a})\times \mathfrak{inn}(\mathfrak{b})$. 

\subsection{(Outer) automorphism group schemes of positively graded Lie algebras}\label{sect:3:2}

Let $\mathfrak a$ be a positively graded complete $\mathbb Q$-Lie algebra, generated in degree $1$ and whose degree 1 component
is finite dimensional. Recall that 
for $\mathbf k$ a commutative $\mathbb Q$-algebra,  $\mathrm{Aut}_1(\mathfrak a\hat\otimes \mathbf k)$
is the group of $\mathbf k$-linear Lie algebra automorphisms of  $\mathfrak a\hat\otimes \mathbf k$, 
filtered and with associated graded the identity. 

\begin{lemma}
The assignment $\mathbf k\mapsto \mathrm{Aut}_1(\mathfrak a\hat\otimes \mathbf k)$ is a prounipotent 
$\mathbb Q$-group scheme. The associated Lie algebra is $\mathfrak{der}_1(\mathfrak a)$, 
which is the set of Lie algebra derivations of $\mathfrak a$, filtered and with associated graded equal to $0$. 
\end{lemma}

\begin{proof}
Let $\mathrm{GL}^+(\mathfrak a\hat\otimes \mathbf k)$ be the group of $\mathbf k$-linear 
automorphisms of  $\mathfrak a\hat\otimes \mathbf k$, filtered and with associated graded the identity. 
Let $\mathfrak{end}^+(\mathfrak a)$ be the set of endomorphisms of $\mathfrak a$, filtered and with associated graded equal to $0$. 
The exponential map sets up a bijection $\mathfrak a\hat\otimes \mathbf k\to \mathrm{GL}^+(\mathfrak a\hat\otimes \mathbf k)$, which 
implies that $\mathbf k\mapsto \mathrm{GL}^+(\mathfrak a\hat\otimes \mathbf k)$ is a prounipotent 
$\mathbb Q$-group scheme with associated Lie algebra $\mathfrak{end}^+(\mathfrak a)$. The exponential map
restricts to a bijection $\mathfrak{der}_1(\mathfrak a)\hat\otimes \mathbf k\to \mathrm{Aut}_1(\mathfrak a\hat\otimes \mathbf k)$, 
which implies the claim. 
\end{proof}

Recall that for $\mathbf k$ a commutative $\mathbb Q$-algebra, the assignment  $g\mapsto \mathrm{Ad}_g$ is a group morphism  
$\mathrm{Ad} : \mathrm{exp}(\mathfrak a\hat\otimes \mathbf k)\to\mathrm{Aut}_1(\mathfrak a\hat\otimes \mathbf k)$
 whose image is a normal subgroup, and that the corresponding quotient is denoted 
 $\mathrm{Out}_1(\mathfrak a\hat\otimes \mathbf k):=\mathrm{Aut}_1(\mathfrak a\hat\otimes \mathbf k)/
 \mathrm{Ad}(\mathrm{exp}(\mathfrak a\hat\otimes \mathbf k))$. 

\begin{lemma}
The assignment
$\mathbf k\mapsto \mathrm{Out}_1(\mathfrak a\hat\otimes \mathbf k)$ is a prounipotent 
$\mathbb Q$-group scheme. The associated Lie algebra is $\mathfrak{out}_1(\mathfrak a)$, which is the 
quotient of $\mathfrak{der}_1(\mathfrak a)$ by the image of the Lie algebra morphisms 
$\mathrm{ad} : \mathfrak a\to \mathfrak{der}_1(\mathfrak a)$, $a\mapsto \mathrm{ad}_a:=(x\mapsto [a,x])$. 
\end{lemma}

\begin{proof}
One checks that for $g\in  \mathrm{exp}(\mathfrak a\hat\otimes \mathbf k)$, 
$\mathrm{Ad}_g\in \mathrm{Aut}_1(\mathfrak a\hat\otimes \mathbf k)$; moreover,  
$\theta\in \mathrm{Aut}_1(\mathfrak a\hat\otimes \mathbf k)$ induces an 
automorphism of $\mathrm{exp}(\mathfrak a\hat\otimes \mathbf k)$ and $\theta\mathrm{Ad}_g\theta^{-1}=\mathrm{Ad}_{\theta(g)}$. 
This implies the first statement. The second statement follows from the fact that for $\mathbf k\mapsto G(\mathbf k)$
a prounipotent $\mathbb Q$-group scheme and for $\mathbf k\mapsto H(\mathbf k)$ a normal prounipotent $\mathbb Q$-subgroup scheme, 
$\mathbf k\mapsto G(\mathbf k)/H(\mathbf k)$ is a prounipotent $\mathbb Q$-group scheme with Lie algebra 
$\mathrm{Lie}G/\mathrm{Lie}(H)$. 
\end{proof}

As $\mathfrak{der}_1(\mathfrak a)\subset \mathfrak{der}(\mathfrak a)$, $\mathfrak{out}_1(\mathfrak a)$ is a Lie algebra of 
$\mathfrak{out}(\mathfrak a)$. 

\subsection{Stabilizer group schemes of outer morphisms of Lie algebras} \label{sec:3:3}

Let $\mathfrak a$ and $\mathfrak b$ be positively graded complete $\mathbb Q$-Lie algebras, generated in degree $1$ and 
whose degree 1 components are finite dimensional. 

Let $\mathbf k$ a commutative $\mathbb Q$-algebra. 
Recall from \S\ref{sect:intro} (paragraph before Thm. \ref{thm:4}) the definition of the sets 
$\mathrm{Hom}(\mathfrak a,\mathfrak b)(\mathbf k)$ and $\mathrm{OutHom}(\mathfrak a,\mathfrak b)(\mathbf k)$,  
of the action of the group $\mathrm{Out}_1(\mathfrak a\hat\otimes \mathbf k)\times \mathrm{Out}_1(\mathfrak b\hat\otimes \mathbf k)$
on this set, and of the map  $\mathrm{Hom}(\mathfrak a,\mathfrak b)(\mathbb Q)\ni \mu\mapsto \overline\mu\in 
\mathrm{OutHom}(\mathfrak a,\mathfrak b)(\mathbb Q)\to \mathrm{OutHom}(\mathfrak a,\mathfrak b)(\mathbf k)$. 

% define $\mathrm{Hom}(\mathfrak a,\mathfrak b)(\mathbf k)$ as the set 
% of $\mathbf k$-linear morphisms of complete Lie algebras $\mathfrak a\hat\otimes\mathbf k\to \mathfrak b\hat\otimes\mathbf k$. 
% The group $\mathrm{exp}(\mathfrak b\hat\otimes\mathbf k)$ acts by post-composition on $\mathrm{Hom}(\mathfrak a,\mathfrak b)(\mathbf k)$, 
% denote by $\mathrm{OutHom}(\mathfrak a,\mathfrak b)(\mathbf k)$ the quotient set 
% $\mathrm{Hom}(\mathfrak a,\mathfrak b)(\mathbf k)/\mathrm{exp}(\mathfrak b\hat\otimes\mathbf k)$.

% The group $\mathrm{Aut}_1(\mathfrak a\hat\otimes\mathbf k)\times\mathrm{Aut}_1(\mathfrak b\hat\otimes\mathbf k)$
% acts on $\mathrm{Hom}(\mathfrak a,\mathfrak b)(\mathbf k)$ by combining pre- and post-composition. This 
% induces an action of $\mathrm{Out}_1(\mathfrak a\hat\otimes \mathbf k)\times \mathrm{Out}_1(\mathfrak b\hat\otimes \mathbf k)$
% on $\mathrm{OutHom}(\mathfrak a,\mathfrak b)(\mathbf k)$. 

% For $\mu\in \mathrm{Hom}(\mathfrak a,\mathfrak b)(\mathbb Q)$, let $\overline\mu$ be its image in 
% $\mathrm{OutHom}(\mathfrak a,\mathfrak b)(\mathbb Q)$, and also by $\overline\mu$ the image of the latter element in 
% $\mathrm{OutHom}(\mathfrak a,\mathfrak b)(\mathbf k)$ for any $\mathbf k$. 

\begin{lemma}\label{lem:31}
The assignment $\mathbf k\mapsto \mathrm{Stab}_{\mathrm{Out}_1(\mathfrak a\hat\otimes \mathbf k)\times 
\mathrm{Out}_1(\mathfrak b\hat\otimes \mathbf k)}(\overline\mu)$ is a prounipotent group functor. The associated 
Lie algebra is $\mathfrak{stab}_{ \mathfrak{out}_1(\mathfrak a)\times \mathfrak{out}_1(\mathfrak b)}(\overline\mu)$. 
\end{lemma}

\begin{proof}
The group $\mathrm{Stab}_{\mathrm{Out}_1(\mathfrak a\hat\otimes \mathbf k)\times 
\mathrm{Out}_1(\mathfrak b\hat\otimes \mathbf k)}(\overline\mu)$ is the image in 
$\mathrm{Out}_1(\mathfrak a\hat\otimes \mathbf k)\times 
\mathrm{Out}_1(\mathfrak b\hat\otimes \mathbf k)$ of 
$\{(\alpha,\beta)\in \mathrm{Aut}_1(\mathfrak a\hat\otimes \mathbf k)\times 
\mathrm{Aut}_1(\mathfrak b\hat\otimes \mathbf k)|\beta\mu=\mu\alpha\}$. 
Recall that $\mathrm{exp}$ sets up a bijection
$\mathfrak{der}_1(\mathfrak a)\hat\otimes\mathbf k\to\mathrm{Aut}_1(\mathfrak a)(\mathbf k)$
and 
$\mathfrak{der}_1(\mathfrak b)\hat\otimes\mathbf k\to\mathrm{Aut}_1(\mathfrak b)(\mathbf k)$, whose inverses
will be denoted $\mathrm{log}$. 
The equation $\beta\mu=\mu\alpha$ for $(\alpha,\beta)\in \mathrm{Aut}_1(\mathfrak a\hat\otimes \mathbf k)\times 
\mathrm{Aut}_1(\mathfrak b\hat\otimes \mathbf k)$ is equivalent to 
$\mathrm{log}\beta\circ \mu=\mu\circ \mathrm{log}\alpha$. Therefore the exponential 
sets up a bijection $\{(\alpha,\beta)\in \mathfrak{der}_1(\mathfrak a)\times 
\mathfrak{der}_1(\mathfrak b)|\beta\mu=\mu\alpha\}\hat\otimes\mathbf k\to 
\{(\alpha,\beta)\in \mathrm{Aut}_1(\mathfrak a\hat\otimes \mathbf k)\times 
\mathrm{Aut}_1(\mathfrak b\hat\otimes \mathbf k)|\beta\mu=\mu\alpha\}$. 
It follows that it also sets up a bijection 
$\mathfrak{stab}_{ \mathfrak{out}_1(\mathfrak a)\times \mathfrak{out}_1(\mathfrak b)}(\mu)\hat\otimes\mathbf k\to 
\mathrm{Stab}_{\mathrm{Out}_1(\mathfrak a\hat\otimes \mathbf k)\times 
\mathrm{Out}_1(\mathfrak b\hat\otimes \mathbf k)}(\overline\mu)$, which implies the result. 
\end{proof}

\section{A stabilizer interpretation of $\mathsf{GRT}_1(\mathbf k)$}
\label{sec:4}

The purpose of this section is, based on the material on §\ref{sec:3} on stabilizers, to prove 
Thm. \ref{thm:la:char} on Lie algebra stabilizers of pairs of morphisms and Cors. \ref{thm:gp:char} and 
\ref{thm:gp:char:sphere} on group stabilizers of the same pairs.

Denote by $D\mapsto[D]$ the projection map $\mathfrak{der}(\mathfrak a)\to\mathfrak{out}(\mathfrak a)$. 
\begin{lemma}\label{lem:24}
The map $\mathfrak G_\inert\to \mathfrak{out}(\mathfrak t_3)$ given by $C\mapsto [D_C]$ is injective and defines a 
Lie algebra morphism  $\mathfrak G_\inert\to \mathfrak{out}_1(\mathfrak t_3)$. 
\end{lemma}

\begin{proof}
The kernel of this map is the set of $C\in\mathfrak G_\inert$ such that $D_C$ is inner. If $C$ is such 
an element and if $X\in \mathfrak t_3$
is such that $D_C=\mathrm{ad}_X$, then one may assume the support of $X$ to be contained in that of $C$, therefore in 
$\{k|k\geq2\}$. Moreover, $[X,t_{12}]=0$, which by Lem. \ref{lem:Ct12} implies $X\in\mathbf kt_{12}
+\mathfrak t_2^{12,3}$, therefore $X$ is supported in degree 1. It follows that $X=0$, therefore $D_C=0$. 
This implies $[C(t_{32},t_{21}),t_{23}]=0$, therefore by Lem. \ref{lem:Ct12}, 
$C(t_{32},t_{21})\in \mathbf kt_{23}+\mathfrak t_2^{1,23}$ therefore $C$ is supported in degree $1$, which implies $C=0$. 
The map $C\mapsto [D_C]$ corestricts to $\mathfrak{out}_1(\mathfrak t_3)$ for degree reasons. It is a Lie algebra 
morphism by Lem. \ref{lem:def:DC}. 
\end{proof}

% \begin{theorem}\label{thm:la:char}
% $\mathfrak{grt}_1$ is equal to the intersection
% \begin{equation}\label{main:int}
% \mathfrak G_\inert\cap
% \mathfrak{stab}_{\mathfrak{out}(\mathfrak t_3)}(\mu_{123})\cap
% \mathfrak{stab}_{\mathfrak{out}(\mathfrak t_3)}(\mu_{124}) 
% \end{equation}
% (intersection in the Lie algebra $\mathfrak{out}(\mathfrak t_3)$), where 
% $\mathfrak G_\inert%\times\mathfrak{out}(\mathfrak t_4)
% $ is the Lie subalgebra defined by Lem. \ref{lem:24} and 
% $\mathfrak{stab}_{\mathfrak{out}(\mathfrak t_3)}(\mu_{123})$
% and $\mathfrak{stab}_{\mathfrak{out}(\mathfrak t_3)}(\mu_{124})$
% are the Lie subalgebras associated to the morphisms $\mu_{123},\mu_{124}:\mathfrak t_3\to\mathfrak t_4$
% %, $\mu_{123} : x\mapsto x^{1,2,3}$ and $\mu_{124} : x\mapsto x^{1,2,4}$ 
% (see Lem. \ref{lem:stab:las} and Thm. \ref{thm:main}). 
% \end{theorem}

\begin{theorem}\label{thm:la:char}
$\mathfrak{grt}_1$ is equal to the intersection
\begin{equation}\label{main:int}
\mathfrak{stab}_{\mathfrak G_\inert\times \mathfrak{out}(\mathfrak t_4)}(\overline\mu_{123})\cap
\mathfrak{stab}_{\mathfrak G_\inert\times \mathfrak{out}(\mathfrak t_4)}(\overline\mu_{124}) 
\end{equation}
(intersection in the Lie algebra $\mathfrak G_\inert\times \mathfrak{out}(\mathfrak t_4)$), where the notation is as in 
Lem. \ref{lem:stab:las}, the inclusion 
$\mathfrak G_\inert\subset\mathfrak{out}(\mathfrak t_3)$ is as in Lem. \ref{lem:24} and 
%$\mathfrak{stab}_{\mathfrak{out}(\mathfrak t_3)\times \mathfrak{out}(\mathfrak t_4)}(\mu_{123})$ and 
%$\mathfrak{stab}_{\mathfrak{out}(\mathfrak t_3)\times \mathfrak{out}(\mathfrak t_4)}(\mu_{124})$ are 
the morphisms $\mu_{123},\mu_{124}:\mathfrak t_3\to\mathfrak t_4$ are as in Thm. \ref{thm:main}. 
\end{theorem}

Note that this result gives an alternative proof of Lem. \ref{lem:dr}. 

\begin{proof}
Denote by $p: \mathfrak G_\inert\times\mathfrak{out}(\mathfrak t_4)\to\mathfrak G_\inert$
and $p' :  \mathfrak G_\inert\times\mathfrak{der}(\mathfrak t_4)\times \mathfrak t_4\to 
\mathfrak G_\inert\times\mathfrak{out}(\mathfrak t_4)$ the canonical projections. Denote by 
$\mathfrak A$ the space defined by \eqref{main:int}. Then  
\begin{align*}
&\mathfrak A=\{(C,F)|\text{for some }E,E'\in \mathfrak{der}(\mathfrak t_4)\text{ above }F,\text{ one has }
E\mu_{123}=\mu_{123}D_C\text{ and }E'\mu_{124}=\mu_{124}D_C\}
\\ & \subset \mathfrak G_\inert\times\mathfrak{out}(\mathfrak t_4). 
\end{align*}
Denoting by $X\in\mathfrak t_4$ an element such that $E'=E+\mathrm{ad}_X$, one obtains 
\begin{align}\label{eq:24mar}
&\mathfrak A=\{(C,F)|\text{for some }E\in \mathfrak{der}(\mathfrak t_4)\text{ above }F\text{ and some }X\in\mathfrak t_4,
\\ & \nonumber \text{ one has }
E\mu_{123}=\mu_{123}D_C\text{ and }(E+\mathrm{ad}_X)\mu_{124}=\mu_{124}D_C\}. 
\end{align}
Set  
$$
\mathfrak B:=\{(C,E,X)|E\mu_{123}=\mu_{123}D_C\text{ and }(E+\mathrm{ad}_X)\mu_{124}=\mu_{124}D_C\}
\subset \mathfrak G_\inert\times\mathfrak{der}(\mathfrak t_4)\times \mathfrak t_4,  
$$
then \eqref{eq:24mar} implies $\mathfrak A=p'(\mathfrak B)$. It follows 
\begin{equation}\label{p:pp'}
p(\mathfrak A)=pp'(\mathfrak B).     
\end{equation} 
Then 
$$
pp'(\mathfrak B)=\{C|\text{for some }E\in\mathfrak{der}(\mathfrak t_4)\text{ and }X\in\mathfrak t_4,\text{ one has }
E\mu_{123}=\mu_{123}D_C\text{ and }(E+\mathrm{ad}_X)\mu_{124}=\mu_{124}D_C\}, 
$$
which together with Thm. \ref{thm:main} implies $pp'(\mathfrak B)=\mathfrak{grt}_1$. Together with \eqref{p:pp'}, this implies 
$p(\mathfrak A)=\mathfrak{grt}_1$. Therefore $p: \mathfrak G_\inert\times\mathfrak{out}(\mathfrak t_4)\to\mathfrak G_\inert$ induces a 
surjective map $\mathfrak A\to \mathfrak{grt}_1$. 

Let us now prove that the restriction of $p$ to $\mathfrak A$ is injective. The kernel of this restriction is the 
intersection of the kernel of $p$ with $\mathfrak A$, which is therefore equal to 
$$
\mathfrak D:=\{F\in\mathfrak{out}(\mathfrak t_4)|\text{for some }E\in \mathfrak{der}(\mathfrak t_4)\text{ above }
F\text{ and some }X\in\mathfrak t_4,
\nonumber \text{ one has }
E\mu_{123}=(E+\mathrm{ad}_X)\mu_{124}=0\}. 
$$
Set 
$$
\mathfrak C:=\{(E,X)\in  \mathfrak{der}(\mathfrak t_4)\times \mathfrak t_4|E\mu_{123}=(E+\mathrm{ad}_X)\mu_{124}=0\}. 
$$
One checks that 
\begin{equation}\label{eq:proj}
    \mathfrak D=p''(\mathfrak C), 
\end{equation}
where $p'' : \mathfrak{der}(\mathfrak t_4)\times \mathfrak t_4
\to\mathfrak{out}(\mathfrak t_4)$ is the canonical projection. 

Then $\mathfrak C$ is a graded module. Let $(E,X)\in\mathfrak C$. Then 
$E\mu_{123}=0$ implies 
\begin{equation}\label{E:X:1}
    E(t_{12})=E(t_{13})=E(t_{23})=0
\end{equation}
and $(E+\mathrm{ad}_X)\mu_{124}=0$ implies 
\begin{equation}\label{E:X:2}
    E(t_{12})=[t_{12},X],E(t_{14})=[t_{14},X], E(t_{24})=[t_{24},X].
\end{equation}
The resulting equality $[t_{12},X]=0$
then implies
\begin{equation}\label{above:a}
X\in\mathbf kt_{12}+\mathfrak t_3^{12,3,4}. 
\end{equation}
The equality $[t_{14},t_{23}]=0$ implies $[[X,t_{14}],t_{23}]+[t_{14},[X,t_{23}]]=0$, which 
together with $[X,t_{23}]=0$ implies $[t_{23},[t_{14},X]]=0$, which by Lem. \ref{lem:kertt} implies
\begin{equation}\label{above:b}
X\in\mathfrak t_3^{1,4,23}+\mathfrak t_3^{2,3,14}. 
\end{equation}
Similarly, $[t_{13},t_{24}]=0$, $[X,t_{13}]=0$ and the analogue of Lem. \ref{lem:kertt} imply 
\begin{equation}\label{above:c}
X\in\mathfrak t_3^{1,3,24}+\mathfrak t_3^{2,4,13}. 
\end{equation}
Assume that $(E,X)$ is homogeneous, let $d$ be its degree. 

If $d>1$, the relations \eqref{above:a}-\eqref{above:c}
imply the existence of $\lambda,\mu,\nu,\alpha,\beta\in\mathfrak t_3$, such that 
\begin{equation}\label{eq:24marbis}
X=\lambda^{12,3,4}=\mu^{1,4,23}+\nu^{2,3,14}=\alpha^{1,3,24}+\beta^{2,4,13}. 
\end{equation}
Applying $\mathfrak t_4\to \mathfrak t_3$, $x\mapsto x^{1,2,\emptyset,3}$ to the second equality in \eqref{eq:24marbis} implies 
$0=\mu^{1,3,2}=\beta^{2,3,1}$ hence $\alpha=\beta=0$, and applying $x\mapsto x^{1,2,3,\emptyset}$ to the same implies
$0=\nu^{2,3,1}$, hence $\nu=0$, which together with the first equality in \eqref{eq:24marbis} implies $X=0$. 
Then \eqref{E:X:1} and \eqref{E:X:2} imply 
\begin{equation}\label{part:concl}
    \text{$E(t_{ij})=0$ for $ij\in\{12,13,14,23,24\}$. }
\end{equation}Moreover, 
as $z_4=\sum_{i,j:1\leq i<j\leq 4}t_{ij}$ is central in $\mathfrak t_4$, $E(z_4)$
is also central in this Lie algebra, with together with $Z(\mathfrak t_4)=\mathbf kz_4$ and the 
fact that $E$ is homogeneous of degree $>1$ implies $E(z_4)=0$, which together with \eqref{part:concl} implies
$E(t_{34})=0$, therefore $E=0$. It follows that the part of degree $d$ of $\mathfrak C$ is $0$ for any $d>1$. 

If $d=0$, then $X=0$, with together with \eqref{E:X:1} and \eqref{E:X:2} implies \eqref{part:concl}. 

If now $d=1$, then  the relations \eqref{above:a} and \eqref{above:b}
imply the existence of $a,b,c,d',e,f,g\in\mathbf k$, such that 
$$
X=at_{12}+bt_{12,3}+ct_{12,4}+d't_{34}=et_{1,23}+ft_{23,4}+gt_{14}+ht_{2,14}+it_{3,14}+jt_{23}. 
$$
This implies
$$
X=et_{123}+ht_{124}+ft_{123,4}+it_{124,3}, 
$$
where $t_{ijk}:=t_{ij}+t_{ik}+t_{jk}$. Since $t_{124}$ and $t_{124,3}$ commute with $t_{14},t_{24},t_{34}$, 
its follows that  
$$
E(t_{12})=E(t_{13})=E(t_{23})=0,\ E(t_{14})=[t_{14},et_{123}+ft_{123,4}],\ 
E(t_{24})=[t_{24},et_{123}+ft_{123,4}],
$$
which since $t_{123}$ and $t_{123,4}$ commute with $t_{13},t_{23},t_{23}$, implies that 
$E=-\mathrm{ad}_{et_{123}+ft_{123,4}}$, therefore that $E$ is inner. All this implies the inclusion 
$$
\mathfrak C\subset \mathfrak{inn}(\mathfrak t_4)\times\mathfrak t_4. 
$$
Together with \eqref{eq:proj}, it implies $\mathfrak D=0$, therefore that the restriction of $p$ to 
$\mathfrak A$ is injective. This ends the proof that $p$ induces an isomorphism 
$\mathfrak A\to\mathfrak{grt}_1$. 
\end{proof}

\begin{remark}
One can prove analogously to Lem. \ref{lem:24} that the map $C\mapsto [\vec D_C]$, where $\vec D_C$ is as in 
Lem. \ref{def:D:vecC}, is an injection $\mathfrak G_\inert\to\mathfrak{out}(\mathfrak P_{\vec 4})$, and 
analogously to Thm. \ref{thm:la:char} the equality of $\mathfrak{grt}_1$ with the joint intersection 
in $\mathfrak{out}(\mathfrak P_{\vec 4})%\times\mathfrak{out}(\mathfrak P_{\vec 5})
$ of %the Lie subalgebra 
$\mathfrak G_\inert%\times\mathfrak{out}(\mathfrak P_{\vec 5})
$ and of the stabilizer Lie subalgebras 
associated with the morphisms $\mathfrak P_{\vec 4}\to \mathfrak P_{\vec 5}$ given by $x\mapsto x^{1,2,3,45}$
and $x\mapsto x^{1,2,4,35}$ (see \S\ref{sec:2:1}). 
\end{remark}

\begin{corollary}\label{thm:gp:char}
For any commutative ring $\mathbf k$ containing $\mathbb Q$, one has 
\begin{equation}%\label{main:int}
\mathsf{GRT}_1(\mathbf k)=\mathrm{Stab}_{\mathrm{exp}(\hat{\mathfrak G}_\inert\hat\otimes\mathbf k)\times 
\mathrm{Out}_1(\hat{\mathfrak t}_4\hat\otimes\mathbf k)}(\overline\mu_{123},\overline\mu_{124}) , 
\end{equation}
 the right-hand side being the stabilizer group relative to the diagonal of action of the subgroup 
$\mathrm{exp}(\hat{\mathfrak G}_\inert\hat\otimes\mathbf k)\times 
\mathrm{Out}_1(\hat{\mathfrak t}_4\hat\otimes\mathbf k)$ of 
$\mathrm{Out}_1(\hat{\mathfrak t}_3\hat\otimes\mathbf k)\times 
\mathrm{Out}_1(\hat{\mathfrak t}_4\hat\otimes\mathbf k)$ on the Cartesian square of 
$\mathrm{OutHom}(\mathfrak t_3,\mathfrak t_4)(\mathbf k)$. 
% (intersection in the Lie algebra $\mathfrak G_\inert\times \mathfrak{out}(\mathfrak t_4)$), where the notation is as in 
% Lem. \ref{lem:stab:las}, the inclusion 
% $\mathfrak G_\inert\subset\mathfrak{out}(\mathfrak t_3)$ is as in Lem. \ref{lem:24} and 
% %$\mathfrak{stab}_{\mathfrak{out}(\mathfrak t_3)\times \mathfrak{out}(\mathfrak t_4)}(\mu_{123})$ and 
% %$\mathfrak{stab}_{\mathfrak{out}(\mathfrak t_3)\times \mathfrak{out}(\mathfrak t_4)}(\mu_{124})$ are 
% the morphisms $\mu_{123},\mu_{124}:\mathfrak t_3\to\mathfrak t_4$ are as in Thm. \ref{thm:main}. 
\end{corollary}

\begin{proof}
This follows from the equality of the said right-hand side with the intersection of the groups 
$\mathrm{Stab}_{\mathrm{exp}(\hat{\mathfrak G}_\inert\hat\otimes\mathbf k)\times 
\mathrm{Out}_1(\hat{\mathfrak t}_4\hat\otimes\mathbf k)}(\overline\mu_{123})$ and 
$\mathrm{Stab}_{\mathrm{exp}(\hat{\mathfrak G}_\inert\hat\otimes\mathbf k)\times 
\mathrm{Out}_1(\hat{\mathfrak t}_4\hat\otimes\mathbf k)}(\overline\mu_{124})$, from the equality of 
these groups with $\mathrm{exp}(\mathfrak{stab}_{\mathfrak G_\inert\times \mathfrak{out}_1(\mathfrak t_4)}(\overline\mu_{123})\hat\otimes\mathbf k)$  and 
$\mathrm{exp}(\mathfrak{stab}_{\mathfrak G_\inert\times \mathfrak{out}_1(\mathfrak t_4)}
(\overline\mu_{124})\hat\otimes\mathbf k)$ respectively (see Lem. \ref{lem:31}), 
and from the equality 
 of $\mathfrak{grt}_1$ with  the intersection
$$
\mathfrak{stab}_{\mathfrak G_\inert\times \mathfrak{out}_1(\mathfrak t_4)}(\overline\mu_{123})\cap
\mathfrak{stab}_{\mathfrak G_\inert\times \mathfrak{out}_1(\mathfrak t_4)}(\overline\mu_{124}) , 
$$
which follows from Thm. \ref{thm:la:char} as $\mathfrak{grt}_1$ is supported in positive degrees. 
\end{proof}

\begin{lemma}
The map $\mathfrak G_\inert\to \mathfrak{out}(\mathfrak P_{\vec 4})$ given by $C\mapsto [\vec D_C]$ 
(see Lem. \ref{def:D:vecC})
is injective and defines a Lie algebra morphism  $\mathfrak G_\inert\to \mathfrak{out}_1(\mathfrak t_3)$. 
\end{lemma}

\begin{proof}
    Similar to the proof of Lem. \ref{lem:24}. 
\end{proof}

\begin{corollary}\label{thm:gp:char:sphere}
For any commutative ring $\mathbf k$ containing $\mathbb Q$, one has 
\begin{equation}%\label{main:int}
\mathsf{GRT}_1(\mathbf k)=\mathrm{Stab}_{\mathrm{exp}(\hat{\mathfrak G}_\inert\hat\otimes\mathbf k)\times 
\mathrm{Out}_1(\hat{\mathfrak P}_{\vec 5}\hat\otimes\mathbf k)}(\overline{\vec\mu}_{123},\overline{\vec\mu}_{124}) , 
\end{equation}
 the right-hand side being the stabilizer group relative to the diagonal of action of the subgroup 
$\mathrm{exp}(\hat{\mathfrak G}_\inert\hat\otimes\mathbf k)\times 
\mathrm{Out}_1(\mathfrak P_{\vec 5}\hat\otimes\mathbf k)$ of 
$\mathrm{Out}_1(\mathfrak P_{\vec 4}\hat\otimes\mathbf k)\times 
\mathrm{Out}_1(\mathfrak P_{\vec 5}\hat\otimes\mathbf k)$ on the Cartesian square of 
$\mathrm{OutHom}(\mathfrak P_{\vec 4},\mathfrak P_{\vec 5})(\mathbf k)$, 
$\vec\mu_{123},\vec\mu_{124}$ being the Lie algebra morphisms from Thm. \ref{thm:main:2}
and $\overline{\vec\mu}_{123},\overline{\vec\mu}_{124}$ being their classes in 
$\mathrm{OutHom}(\mathfrak P_{\vec 4},\mathfrak P_{\vec 5})(\mathbf k)$.  
\end{corollary}

\begin{proof}
    Similar to that of Cor. \ref{thm:gp:char}. 
\end{proof}

%\begin{thebibliography}{BGFr}

\end{document}